\documentclass[12pt]{amsart}
\usepackage{amsmath,amsthm,amsfonts,amssymb,mathrsfs}
\date{\today}

\usepackage{color}
\input xymatrix
\xyoption{all}

\usepackage{hyperref}

\newtheorem{theorem}{Theorem}[section]

\newtheorem{proposition}[theorem]{Proposition}
\newtheorem{corollary}[theorem]{Corollary}
\newtheorem{lemma}[theorem]{Lemma}
\theoremstyle{definition}
\newtheorem{example}[theorem]{Example}%[section]
\newtheorem{remark}[theorem]{Remark}%[section]

\newtheorem{definition}[theorem]{Definition}%[section]

\begin{document}

\title[On feebly compact inverse primitive (semi)topological semigroups]{On feebly compact inverse primitive (semi)topological semigroups}

\author{Oleg~Gutik}
\address{Faculty of Mathematics, National University of Lviv,
Universytetska 1, Lviv, 79000, Ukraine}
\email{o\underline{\hskip5pt}\,gutik@franko.lviv.ua,
ovgutik@yahoo.com}

\author{Oleksandr~Ravsky}
\address{Pidstrygach Institute for Applied Problems of Mechanics and Mathematics of NASU, Naukova 3b, Lviv, 79060, Ukraine}
\email{oravsky@mail.ru}



\begin{abstract}
We study the structure of inverse primitive feebly compact semitopological and topological semigroups. We find conditions when the maximal subgroup of an inverse primitive feebly compact semitopological semigroup $S$ is a closed subset of $S$ and describe the topological structure of such semiregular semitopological semigroups. Later we describe the structure of feebly compact topological Brandt $\lambda^0$-extensions of topological semigroups and semiregular (quasi-regular) primitive inverse topological semigroups. In particular we show that inversion in a quasi-regular primitive inverse feebly compact topological semigroup is continuous. Also an analogue of Comfort--Ross Theorem is proved for such semigroups: a Tychonoff product of an arbitrary family of primitive inverse semiregular feebly compact semitopological semigroups with closed maximal subgroups is feebly compact. We describe the structure of the Stone-\v{C}ech compactification of a Hausdorff primitive inverse countably compact semitopological semigroup $S$ such that every maximal subgroup of $S$ is a topological group.
\end{abstract}

\maketitle

\section{Introduction and preliminaries}

Further we shall follow the terminology of \cite{CHK, CliffordPreston1961-1967, Engelking1989, Petrich1984, Ruppert1984}. By $\mathbb{N}$ we shall denote the set of all positive integers.

A semigroup is a non-empty set with a binary associative operation. A semigroup $S$ is called \emph{inverse} if for any $x\in S$ there exists a unique $y\in S$ such that $x\cdot y\cdot x=x$ and $y\cdot x\cdot y=y$. Such the element $y$ in $S$ is called
\emph{inverse} to $x$ and is denoted by $x^{-1}$. The map assigning to each element $x$ of an inverse semigroup $S$ its inverse $x^{-1}$ is called the \emph{inversion}.

For a semigroup $S$ by $E(S)$ we denote the subset of idempotents of $S$, and by $S^1$ (resp., $S^0$) we denote the semigroup $S$ with the adjoined unit (resp., zero) (see \cite[Section~1.1]{CliffordPreston1961-1967}). Also if a
semigroup $S$ has zero $0_S$, then for any $A\subseteq S$ we
denote $A^*=A\setminus\{ 0_S\}$.

For  a semilattice $E$  the semilattice operation on $E$ determines the partial order $\leqslant$ on $E$: $$e\leqslant f\quad\text{if and only if}\quad ef=fe=e.$$ This order is called {\em natural}. An element $e$ of a partially ordered set $X$ is
called {\em minimal} if $f\leqslant e$  implies $f=e$ for $f\in X$. An idempotent $e$ of a semigroup $S$ without zero (with zero) is called \emph{primitive} if $e$ is a minimal element in $E(S)$ (in $(E(S))^*$).

Let $S$ be a semigroup with zero and $\lambda\geqslant 1$ be a cardinal. On the set
$B_{\lambda}(S)=\left(\lambda\times S\times\lambda\right)\sqcup\{ 0\}$ we define a  semigroup operation as follows
 $$
 (\alpha,a,\beta)\cdot(\gamma, b, \delta)=
  \begin{cases}
    (\alpha, ab, \delta), & \text{ if } \beta=\gamma; \\
    0, & \text{ if } \beta\ne \gamma,
  \end{cases}
 $$
and $(\alpha, a, \beta)\cdot 0=0\cdot(\alpha, a, \beta)=0\cdot 0=0$, for all $\alpha, \beta, \gamma, \delta\in \lambda$ and $a, b\in S$. If $S$ is a monoid, then the semigroup $B_\lambda(S)$ is called the {\it Brandt $\lambda$-extension of the semigroup}
$S$~\cite{Gutik1999}. Obviously, ${\mathcal J}=\{ 0\}\cup\{(\alpha, {\mathscr O}, \beta)\colon {\mathscr O}$ is the zero of $S\}$ is an ideal of $B_\lambda(S)$. We put $B^0_\lambda(S)=B_\lambda(S)/{\mathcal J}$ and we shall call $B^0_\lambda(S)$ the {\it Brandt $\lambda^0$-extension of the semigroup $S$ with zero}~\cite{GutikPavlyk2006}. Further, if $A\subseteq S$ then we shall denote $A_{\alpha,\beta}=\{(\alpha, s, \beta)\colon s\in A \}$ if $A$ does not contain zero, and $A_{\alpha,\beta}=\{(\alpha, s, \beta)\colon s\in A\setminus\{ 0\} \}\cup \{ 0\}$ if $0\in A$, for $\alpha, \beta\in {\lambda}$. If $\mathcal{I}$ is a trivial  semigroup (i.e., $\mathcal{I}$ contains only one element), then by ${\mathcal{I}}^0$ we denote the semigroup $\mathcal{I}$ with the adjoined zero. Obviously, for any
$\lambda\geqslant 2$ the Brandt $\lambda^0$-extension of the semigroup ${\mathcal{I}}^0$ is isomorphic to the semigroup of $\lambda\times\lambda$-matrix units and any Brandt $\lambda^0$-extension of a semigroup with zero contains the
semigroup of $\lambda\times\lambda$-matrix units. Further by $B_\lambda$ we shall denote the semigroup of $\lambda\times\lambda$-matrix units and by $B^0_\lambda(1)$  the subsemigroup of $\lambda\times\lambda$-matrix units of the Brandt $\lambda^0$-extension of a monoid $S$ with zero.

A semigroup $S$ with zero is called \emph{$0$-simple} if $\{0\}$ and $S$ are its only ideals and $S\cdot S\neq\{0\}$, and \emph{completely $0$-simple} if it is $0$-simple and has a primitive idempotent~\cite{CliffordPreston1961-1967}. A completely $0$-simple inverse semigroup is called a \emph{Brandt semigroup}~\cite{Petrich1984}. By  Theorem~II.3.5~\cite{Petrich1984}, a semigroup $S$ is a Brandt semigroup if and only if $S$ is isomorphic to a Brandt $\lambda$-extension $B_\lambda(G)$ of a group  $G$.

Let $\left\{S_{\iota}\colon \iota\in\mathscr{I}\right\}$ be a disjoint family of semigroups with zero such that $0_\iota$ is zero in $S_{\iota}$ for any $\iota\in\mathscr{I}$. We put $S=\{0\}\cup\bigcup\left\{S_{\iota}^*\colon \iota\in\mathscr{I}\right\}$, where $0\notin\bigcup\left\{S_{\iota}^*\colon \iota\in\mathscr{I}\right\}$, and define a semigroup  operation ``\;$\cdot$\;'' on $S$ in the following way
\begin{equation*}
    s\cdot t=
\left\{
  \begin{array}{cl}
    st, & \hbox{if } st\in S_{\iota}^* \hbox{ for some } \iota\in\mathscr{I};\\
    0, & \hbox{otherwise}.
  \end{array}
\right.
\end{equation*}
The semigroup $S$ with the operation ``\;$\cdot$\;'' is called an \emph{orthogonal sum} of the semigroups $\left\{S_{\iota}\colon \iota\in\mathscr{I}\right\}$ and in this case we shall write $S=\sum_{\iota\in\mathscr{I}}S_{\iota}$.

A non-trivial inverse semigroup is called a \emph{primitive inverse semigroup} if all its non-zero idempotents are primitive~\cite{Petrich1984}. A semigroup $S$ is a primitive inverse semigroup if and only if $S$ is an orthogonal sum of Brandt semigroups~\cite[Theorem~II.4.3]{Petrich1984}. We shall call a Brandt subsemigroup $T$ of a primitive inverse semigroup $S$ \emph{maximal} if every Brandt subsemigroup of $S$ which contains $T$, coincides with $T$.

In this paper all topological spaces are Hausdorff. If $Y$ is a subspace of a topological space $X$ and $A\subseteq Y$, then by $\operatorname{cl}_Y(A)$ and $\operatorname{int}_Y(A)$ we denote the topological closure and interior of $A$ in $Y$, respectively.

A subset $A$ of a topological space $X$ is called \emph{regular open} if $\operatorname{int}_X(\operatorname{cl}_X(A))=A$.

We recall that a topological space $X$ is said to be
\begin{itemize}
  \item \emph{semiregular} if $X$ has a base consisting of regular open subsets;
  \item \emph{quasiregular} if for any non-empty open set $U\subset X$ there exists a non-empty open set $V\subset U$ such that $\operatorname{cl}_X(V) \subseteq U$;
  \item \emph{compact} if each open cover of $X$ has a finite subcover;
  \item \emph{sequentially compact} if each sequence $\{x_i\}_{i\in\mathbb{N}}$ of $X$ has a convergent subsequence in $X$;
  \item \emph{countably compact} if each open countable cover of $X$ has a finite subcover;
  \item \emph{countably compact at a subset} $A\subseteq X$ if every infinite subset $B\subseteq A$  has  an  accumulation  point $x$ in $X$;
  \item \emph{countably pracompact} if there exists a dense subset $A$ in $X$  such that $X$ is countably compact at $A$;
  \item \emph{feebly compact} if each locally finite open cover of $X$ is finite;
  \item \emph{pseudocompact} if $X$ is Tychonoff and each continuous real-valued function on $X$ is bounded;
  \item $k$-\emph{space} if a subset $F\subset X$ is closed in $X$ if and only if $F\cap K$ is closed in $K$ for every compact subspace $K\subseteq X$.
\end{itemize}
According to Theorem~3.10.22 of \cite{Engelking1989}, a Tychonoff topological space $X$ is feebly compact if and only if $X$ is Pseudocompact. Also, a Hausdorff topological space $X$ is feebly compact if and only if every locally finite family of non-empty open subsets of $X$ is finite. Every compact space and every sequentially compact space are countably compact, every countably compact space is countably pracompact, and every countably pracompact space is feebly compact (see \cite{Arkhangelskii1992}).

We recall that the Stone-\v{C}ech compactification of a Tychonoff space $X$ is a
compact Hausdorff space $\beta X$ containing $X$ as a dense subspace so that each continuous map $f\colon X\rightarrow Y$ to a compact Hausdorff space $Y$ extends to a continuous map $\overline{f}\colon \beta X\rightarrow Y$ \cite{Engelking1989}.

A ({\it semi})\emph{topological semigroup} is a Hausdorff topological space with a (separately) continuous semigroup operation. A topological semigroup which is an inverse semigroup is called an \emph{inverse topological semigroup}. A \emph{topological inverse semigroup} is an inverse topological semigroup with continuous inversion. We observe that the inversion on a topological inverse semigroup is a homeomorphism (see \cite[Proposition~II.1]{EberhartSelden1969}). A Hausdorff topology $\tau$ on a (inverse) semigroup $S$ is called (\emph{inverse}) \emph{semigroup} if $(S,\tau)$ is a topological (inverse) semigroup. A {\it paratopological} (\emph{semitopological}) \emph{group} is a Hausdorff topological space with a jointly (separately) continuous group operation. A paratopological group with continuous inversion is a \emph{topological group}.

Let $\mathfrak{STSG}_0$ be a class of semitopological semigroups.

\begin{definition}[\cite{Gutik1999}]\label{def1}
Let $\lambda\geqslant 1$ be a cardinal and $(S,\tau)\in\mathfrak{STSG}_0$ be a semitopological monoid with zero. Let $\tau_{B}$ be a topology on $B_{\lambda}(S)$ such that
\begin{itemize}
    \item[a)] $\left(B_{\lambda}(S),\tau_{B}\right)\in \mathfrak{STSG}_0$; \; and
    \item[b)] for some $\alpha\in{\lambda}$ the topological subspace $(S_{\alpha,\alpha},\tau_{B}|_{S_{\alpha,\alpha}})$ is naturally homeomorphic to $(S,\tau)$.
\end{itemize}
Then $\left(B_{\lambda}(S), \tau_{B}\right)$ is called a {\it topological Brandt $\lambda$-extension of $(S, \tau)$ in $\mathfrak{STSG}_0$}.
\end{definition}

\begin{definition}[\cite{GutikPavlyk2006}]\label{def2}
Let $\lambda\geqslant 1$ be a cardinal and $(S,\tau)\in\mathfrak{STSG}_0$. Let $\tau_{B}$ be a topology on $B^0_{\lambda}(S)$ such that
\begin{itemize}
  \item[a)] $\left(B^0_{\lambda}(S),            \tau_{B}\right)\in\mathfrak{STSG}_0$;
  \item[b)] the topological subspace $(S_{\alpha,\alpha},\tau_{B}|_{S_{\alpha,\alpha}})$ is naturally homeomorphic to $(S,\tau)$ for some $\alpha\in{\lambda}$.
\end{itemize}
Then $\left(B^0_{\lambda}(S), \tau_{B}\right)$ is called a {\it topological Brandt $\lambda^0$-extension of $(S, \tau)$ in $\mathfrak{STSG}_0$}.
\end{definition}

Later, if $\mathfrak{STSG}_0$ coincides with the class of all semitopological semigroups we shall say that $\left(B^0_{\lambda}(S), \tau_{B}\right)$ (resp.,  $\left(B_{\lambda}(S), \tau_{B}\right)$) is a {\it topological Brandt $\lambda^0$-extension} (resp., a \emph{topological Brandt $\lambda$-extension}) of $(S, \tau)$.

Algebraic properties of Brandt $\lambda^0$-extensions of monoids with zero, non-trivial homomorphisms between them, and a category whose objects are ingredients of the construction of such extensions were described in \cite{GutikRepovs2010}. Also, in  \cite{GutikPavlykReiter2009} and \cite{GutikRepovs2010} a category whose objects are ingredients in the constructions of finite (resp., compact, countably compact) topological Brandt $\lambda^0$-extensions of topological monoids with zeros were described.

Gutik and Repov\v{s} proved that any $0$-simple countably compact topological inverse semigroup is topologically isomorphic to a topological Brandt $\lambda$-extension $B_{\lambda}(H)$ of a countably compact topological group $H$ in the class of all topological inverse semigroups for some finite cardinal $\lambda\geqslant 1$ \cite{GutikRepovs2007}.  Also, every $0$-simple feebly compact topological inverse semigroup is topologically isomorphic to a topological Brandt $\lambda$-extension $B_{\lambda}(H)$ of a feebly compact topological group $H$ in the class of all topological inverse semigroups for some finite cardinal $\lambda\geqslant 1$ \cite{GutikPavlykReiter2011}. Next Gutik and Repov\v{s} showed in \cite{GutikRepovs2007} that the Stone-\v{C}ech compactification $\beta(T)$ of a $0$-simple countably compact topological inverse semigroup $T$ has a natural structure of a $0$-simple compact topological inverse semigroup. It was proved in \cite{GutikPavlykReiter2011} that the same is true for $0$-simple feebly compact topological inverse semigroups.

In the paper \cite{BerezovskiGutikPavlyk2010} the structure of compact and countably compact primitive topological inverse semigroups was described and was shown that any countably compact primitive topological inverse semigroup embeds into a compact primitive topological inverse semigroup.

Comfort and Ross in \cite{ComfortRoss1966} proved that a Tychonoff product of an arbitrary family of pseudocompact topological groups is a pseudocompact topological group. Also, they proved there that the Stone-\v{C}ech compactification of a pseudocompact topological group has a natural structure of a compact topological group. Ravsky in \cite{Ravsky-arxiv1003.5343v5} generalized Comfort--Ross Theorem and proved that a Tychonoff product of an arbitrary non-empty family of feebly compact paratopological groups is feebly compact.

In the paper \cite{GutikPavlyk2013??} the structure of feebly compact primitive topological inverse semigroups is described and it is shown that the Tychonoff product of an arbitrary non-empty family of feebly compact primitive topological inverse semigroups is feebly compact. Also, it is proved that the Stone-\v{C}ech compactification of a feebly compact primitive topological inverse semigroup has a natural structure of a compact primitive topological inverse semigroup.

In this paper we study the structure of inverse primitive feebly compact semitopological and topological semigroups. We find conditions when a maximal subgroup of an inverse primitive feebly compact semitopological semigroup $S$ is a closed subset of $S$ and describe the topological structure of such semiregular semigroup. Later we describe the structure of feebly compact topological Brandt $\lambda^0$-extensions of topological semigroups and semiregular (quasi-regular) primitive inverse topological semigroups. In particular we show that the inversion in a quasi-regular primitive inverse feebly compact topological semigroup is continuous. Also an analogue of Comfort--Ross Theorem is proved for such semigroups: the Tychonoff product of an arbitrary family of primitive inverse semiregular feebly compact semitopological semigroups with closed maximal subgroups is a feebly compact space. We describe the structure of the Stone-\v{C}ech compactification of a Tychonoff primitive inverse countably compact semitopological semigroup $S$ such that every maximal subgroup of $S$ is a topological group.
%%%%%%%%%%%%%%%%%%%%%%%%%%%%%%%%%%%%%%%%%%%%%%%%%%%%%

\section{An adjunction of zero to a compact like semitopological group}

Given a topological space $(X,\tau)$ Stone \cite{Sto} and Kat\u{e}tov \cite{Kat} consider the topology $\tau_r$ on $X$ generated by the base consisting of all regular open sets of the space $(X,\tau)$. This topology is called the {\it semiregularization} of the topology $\tau$. If $(X,\tau)$ is a paratopological group then $(X,\tau_r)$ is a $T_3$ paratopological group \cite[Ex. 1.9]{Rav2}, \cite[p. 31]{Rav3}, and \cite[p. 28]{Rav3}.

\begin{lemma}[{\cite[Theorem 1.7]{ArkRez}}]\label{ArkRezReg}
Each paratopological group that is a dense $G_\delta$-subset of a regular feebly compact space is a topological group.
\end{lemma}

We recall that a group $G$ endowed with a topology is \emph{left} (resp. \emph{right}) ($\omega$-)\emph{precompact}, if for each neighborhood $U$ of the unit of $G$ there exists a (countable) finite subset $F$ of $G$ such that $FU=G$ (resp. $UF=G$). It is easy to check (see, for instance, \cite[Proposition~3.1]{Rav2} or \cite[Proposition~2.1]{Rav2}) that a paratopological group $G$ is left precompact if and only if $G$ is right precompact, so we shall call left precompact paratopological groups as precompact. Moreover, it is well known \cite{AT} that a Hausdorff topological group $G$ is precompact if and only if $G$ is a subgroup of a compact topological group. Theorem 1 from \cite{BGR} implies the following:

\begin{lemma}\label{TBGr}
A Hausdorff topological group $G$ is precompact if and only if
for any neighborhood $W$ of the unit of the group $G$ there exists a finite set
$F\subset G$ such that $G=FWF$.
\end{lemma}

\begin{lemma}\label{lemma-2.1}
Let $S$ be a Hausdorff left topological semigroup, $0$ be a right zero of the semigroup $S$ and $G=S\setminus\{0\}$ be a subgroup of the semigroup $S$. Then $0$ is an isolated point of the semigroup $S$ provided one of the following conditions holds:
\begin{itemize}
  \item[(1)] the group $G$ is left precompact;
  \item[(2)] the group $G$ is a feebly compact paratopological group;
  \item[(3)] the group $G$ is left $\omega$-precompact and feebly compact;
  \item[(4)] $S$ is a feebly compact topological semigroup;
  \item[(5)] $S$ is a topological semigroup and for each neighbourhood $U\subset G$ of the unit of the group $G$ there exists a finite subset $F$ of the group $G$ such that $G=FU^{-1}U$.
\end{itemize}
\end{lemma}

\begin{proof}
Assume the contrary. Put $\mathcal F=\{U\cap G\colon U\subset S$ is a neighbourhood of the point $0\}$. Since $0$ is a non-isolated point of the semigroup $S$, the family $\mathcal F$ is a filter. Let $x\in G$ be an arbitrary element and $U$ be an arbitrary member of the filter $\mathcal F$. Since $x0=0$ and left shifts on the semigroup $S$ are continuous, there exists a member $V$ of the filter $\mathcal F$ such that $xV\subset U$. Then $V\subset x^{-1}U$, so $x^{-1}U\in\mathcal{F}$. Since $S$ is Hausdorff, there exists a neighbourhood $W\subset G$ of the unit such that
$G\setminus W\in\mathcal F$.

Now we consider cases (1)--(5) separately.

(1) Since the group $G$ is left precompact, there exists a finite subset
$F$ of the group $G$ such that $FW=G$. But then
\begin{equation*}
\varnothing=G\setminus\bigcup_{x\in F}xW=
\bigcap_{x\in F} x(G\setminus W)\in\mathcal F,
\end{equation*}
a contradiction.

(2) Since the semiregularization $G_r$ of the group $G$ is a feebly compact $T_3$ (and, hence, a regular)
paratopological group, $G_r$ is a topological group by Lemma~\ref{ArkRezReg}.
Therefore $G_r$ is precompact. Thus there exists a finite subset
$F$ of the group $G$ such that $F\cdot\operatorname{cl}_G(W)=G$.
But then
\begin{equation*}
\varnothing=G\setminus\bigcup_{x\in F}x\cdot\operatorname{cl}_G(W)=
\bigcap_{x\in F} x(G\setminus\operatorname{cl}_G(W))\in\mathcal F,
\end{equation*}
a contradiction.

(3) Since the group $G$ is left $\omega$-precompact, there exists a countable subset
$C=\{c_n\colon n\in\mathbb{N}\}$ of the group $G$ such that $CW=G$. For each positive integer $n$ put $C_n=\{c_i\colon 1\le i\le n\}$ and $V_n=G\setminus C_nW$. Since the family $\mathcal F$ is a filter we have that $V_n\in\mathcal F$. Since $0$ is a non-isolated point of the semigroup $S$,
 $\operatorname{int}_G(V_n)$ is a non-empty open subset of the
space $G$. Since the space $G$ is feebly compact, there exists a point $x\in\bigcap_{n\in\mathbb{N}}\operatorname{cl}_G\left(\operatorname{int}_G(V_n)\right)$.
Since $G=CW$ we conclude that there exists a positive integer $n$ such that $x\in c_nW$. But
\begin{equation*}
c_nW\cap\operatorname{cl}_G\left(\operatorname{int}_G(V_n)\right)\subset c_nW\cap\operatorname{cl}_G(V_n)=c_nW\cap\operatorname{cl}_G\left(G\setminus C_nW\right)=c_nW\cap (G\setminus C_nW)=\varnothing,
\end{equation*}
a contradiction.

(4) First we suppose that the space of the semigroup $S$ is regular.
Lemma~\ref{ArkRezReg} implies that
$G$ is a topological group.
If the group $G$ is left precompact then $0$ is an isolated point of the semigroup $S$ by Case (1). So we assume the group $G$ is not left precompact. By Lemma~\ref{TBGr} there exists a neighbourhood
$W_0\subset G$ of the unit such that $G\not=F_0W_0F_0$ for each finite subset $F_0$ of the group $G$. The multiplication on the semigroup $S$ is continuous. Hence there exists a member $V_1$ of the filter $\mathcal F$ such that $V_1^2\subset G\setminus W$. Moreover, there exist a symmetric open neighbourhood
$W_1$ of the unit and a member $V_2$ of the filter $\mathcal F$ such that
$W_1^5V_2\subset V_1$ and $W_1^4\subset W_0$.
Let $C$ be a maximal subset of the set $G\setminus V_2$ such that $W_1^2c\cap W_1^2c'=\varnothing$
for distinct elements $c,c'$ of the set $C$. If $z$ is an arbitrary element of the set $G\setminus V_2$
then $W_1^2c\cap W_1^2z\not=\varnothing$ for an element $c$ of the set $C$. Hence
$G\setminus V_2\subset W_1^4C$. Put $F=\{c\in C\colon W_1c\cap V_2=\varnothing\}$. Then we have that $C\setminus F\subset W_1V_2$ and hence $G\setminus V_2\subset W_1^4C\subset W_1^4F\cup W_1^5V_2$. Then we get that
$G\setminus V_1\subset G\setminus V_2\subset W_1^4F\cup W_1^5V_2$ and hence $G\setminus V_1\subset W_1^4F$, because $W_1^5V_2\subset V_1$. Since $e\not\in G\setminus W\supset V_1^2\supset (G\setminus W_1^4F)^2$, we see that $x(G\setminus W_1^4F)\not\ni e$ for each element $x\in G\setminus W_1^4F$. Then we have that $(G\setminus W_1^4F)^{-1}\subset W_1^4F$ and hence $G\subset W_1^4F\cup F^{-1}W_1^4$.

Since $W_1^4\subset W_0$ we conclude that the set $F$ is infinite. Let $C^{\prime}$ be an arbitrary countable infinite subset of the set $F$. Since the space $S$ is feebly compact we have that there exists a point $x_0\in S$ such that each neighbourhood $V^{\prime}$ of the point $x_0$ intersects infinitely many members of the family $\left\{W_1c\colon c\in C^{\prime}\right\}$ of the open subsets of the space $S$. Clearly, $x_0\not=0$. Then $x_0\in G$. Put $V^{\prime}=W_1x_0$. Then there exist distinct elements $c$ and $c^{\prime}$ of the set $C^{\prime}$ such that $W_1c\cap W_1x_0\not=\varnothing$ and $W_1c^{\prime}\cap W_1x_0\not=\varnothing$. This implies $x_0\in W_1^2c\cap W_1^2c^{\prime}\neq\varnothing$, a contradiction.

Now we consider the case when the space of the semigroup $S$ is not necessarily
regular. We claim that the semiregularization $S_r$ of the semigroup $S$ is a regular
topological semigroup.

Indeed, let $U=\operatorname{int}_S(\operatorname{cl}_S(U))$ be an arbitrary
regular open subset of the space $S$ and $x\in U$ be an arbitrary point. If $x\not=0$ then there exists
an open neighbourhood $W\subset G$ of the unit such that $0\not\in  \operatorname{cl}_S(W)$ and $xW^2\subset U$. Then $x\in xW^2\subset xW\operatorname{cl}_S(W)\subset \operatorname{cl}_S(U)$. Since translations by elements of the group $G$ are homeomorphisms of the space, the set $xW\operatorname{cl}_S(W)$ is open, and hence
\begin{equation*}
x\in xW\subset \operatorname{cl}_S(xW)\subset xW\operatorname{cl}_S(W)\subset\operatorname{int}_S(\operatorname{cl}_S(U)).
\end{equation*}
If $x=0$ then there exist an open neighbourhood $W\subset G$ of the unit and
an open neighbourhood $V\subset S$ of $x$ such that $WV\subset U$. Then $x\in V\subset WV\subset
W\operatorname{cl}_S(V)\subset \operatorname{cl}_S(U)$. We have that $x\in V\subset \operatorname{int}_S(\operatorname{cl}_S(U))$. Let $y\in\operatorname{cl}_S(V)$ be an arbitrary point distinct from $0$. Then $Wy\subset\operatorname{cl}_S(U)$ is an open neighbourhood of $y$. Hence $y\in Wy\subset \operatorname{int}_S(\operatorname{cl}_S(U))$. Therefore the space $S_r$ is regular.

Now we show that multiplication on the semigroup $S_r$ is continuous. Indeed, let $x,y\in S$ be arbitrary points and $O_{xy}=\operatorname{int}_S(\operatorname{cl}_S(O_{xy}))\ni xy$ be an arbitrary regular open subset of the space $S$. There exist open subsets $O_x\ni x$, $O_y\ni y$ of the semigroup $S$ such that $O_xO_y\subset O_{xy}$. Since multiplication on the semigroup $S$ is continuous, $\operatorname{cl}_S(O_x)\cdot\operatorname{cl}_S(O_y)\subset \operatorname{cl}_S(O_{xy})$. Let $x'\in \operatorname{cl}_S(O_x)$, $y'\in \operatorname{cl}_S(O_y)$ be arbitrary points. If $x'\not=0$ then since left translations by elements of the group $G$ are homeomorphisms of $S$ onto itself, the set $x'\operatorname{int}_S(\operatorname{cl}_S(O_y))$ is open, so $x'y'\in\operatorname{int}_S(\operatorname{cl}_S(O_{xy}))$.
Similarly, if $y'\not=0$ then  $x'y'\in\operatorname{int}_S(\operatorname{cl}_S(O_{xy}))$ too.
If not $x=y=0$ then we can choose the neighbourhoods $O_x$ and $O_y$ so small
that $\operatorname{cl}_S(O_x)\cap \operatorname{cl}_S(O_x)\not\ni 0$. Then necessarily
$x'\ne 0$ or $y'\ne 0$. If $x=y=0$ and $x'=y'=0$ then
$x'y'=xy\in \operatorname{int}_S(\operatorname{cl}_S(O_{xy}))$ by the choice of the neighbourhood $O_{xy}$.
Therefore, in all cases we have $x'y'\in\operatorname{int}_S(\operatorname{cl}_S(O_{xy}))$. Thus
$\operatorname{int}_S(\operatorname{cl}_S(O_{x}))\cdot
\operatorname{int}_S(\operatorname{cl}_S(O_{y}))\subset
\operatorname{int}_S(\operatorname{cl}_S(O_{xy}))$.

%The case $x'=y'=0$ can be avoided if only not $x=y=0$. But in the latter case $x'=x=y'=y=0$ and
%$x'y'=xy\in \operatorname{int}_S(\operatorname{cl}_S(O_{xy}))$ by the choice of the neighbourhood $O_{xy}$.

So, by the already proved case of the regular semigroup, $0$ is an isolated point
of the semigroup $S_r$. Since the topology of the semigroup $S_r$ is weaker than
the topology of the semigroup $S$, $0$ is an isolated point of the semigroup $S$.

(5) Since multiplication on the semigroup $S$ is continuous, there exist a neighbourhood $W_1\subset W$ of the unit and a member $V$ of the filter $\mathcal F$ such that $W_1V\subset G\setminus W$. Then $W_1V\cap W=\varnothing$, so $V\cap W_1^{-1}W_1 \subset V\cap W_1^{-1}W=\varnothing$. Hence $G\setminus  W_1^{-1}W_1\in\mathcal F$. By the assumption, there exists a finite subset $F$ of the group $G$ such that $G=FW^{-1}W$. Then
\begin{equation*}
\mathcal F\ni \bigcap_{x\in F} x(G\setminus W_1^{-1}W_1)=G\setminus\bigcup_{x\in
F}xW_1^{-1}W\neq\varnothing,
\end{equation*}
a contradiction.
\end{proof}

\begin{remark}
Authors do not know, if a counterpart of Lemma~\ref{lemma-2.1} holds when the group $G$ is a countably compact semitopological group.
\end{remark}

%%%%%%%%%%%%%%%%%%%%%%%%%%%%%%%%%%%%%%%%%%%%%%%%%%%%%

\section{Feebly compact topological Brandt $\lambda^0$-extensions of topological semigroups and primitive inverse semitopological semigroups}

\begin{proposition}\label{proposition-2.2}
Let $S$ be a Hausdorff semitopological semigroup such that $S$ is an orthogonal sum of the family $\{B_{\lambda_{i}}^0(S_i)\colon i\in\mathscr{I}\}$ of topological Brandt $\lambda^0_i$-extensions of semitopological monoids with zeros. Then for every non-zero element $(\alpha_i,g_i,\beta_i)\in (S_i)_{\alpha_i,\beta_i}\subseteq B_{\lambda_{i}}^0(S_i)\subseteq S$ there exists an open neighbourhood $U_{(\alpha_i,g_i,\beta_i)}$ of $(\alpha_i,g_i,\beta_i)$ in $S$ such that $U_{(\alpha_i,g_i,\beta_i)}\subseteq (S_i)_{\alpha_i,\beta_i}^*$ and hence every set $(S_i)_{\alpha_i,\beta_i}^*$ is an open subset of $S$.
\end{proposition}

\begin{proof}
Suppose the contrary that $U_{(\alpha_i,g_i,\beta_i)}\nsubseteq (S_i)_{\alpha_i,\beta_i}^0$ for every open neighbourhood $U_{(\alpha_i,g_i,\beta_i)}$ of $(\alpha_i,g_i,\beta_i)$ in $S$. Hausdorffness of $S$ implies that there exists an open neighbourhood $V_{(\alpha_i,g_i,\beta_i)}$ of $(\alpha_i,g_i,\beta_i)$ in $S$ such that $0\notin V_{(\alpha_i,g_i,\beta_i)}$. By the separate continuity of multiplication in $S$ there exists an open neighbourhood $W_{(\alpha_i,g_i,\beta_i)}$ of $(\alpha_i,g_i,\beta_i)$ in $S$ such that
\begin{equation*}
W_{(\alpha_i,g_i,\beta_i)}\cdot (\beta_i,e_i,\beta_i)\subseteq  V_{(\alpha_i,g_i,\beta_i)}\qquad \hbox{and} \qquad (\alpha_i,e_i,\alpha_i)\cdot W_{(\alpha_i,g_i,\beta_i)}\subseteq  V_{(\alpha_i,g_i,\beta_i)}.
\end{equation*}
Then condition $W_{(\alpha_i,g_i,\beta_i)}\nsubseteq (S_i)_{\alpha_i,\beta_i}^*$ implies that either $W_{(\alpha_i,g_i,\beta_i)}\cdot (\beta_i,e_i,\beta_i)\ni 0$ or $(\alpha_i,e_i,\alpha_i)\cdot W_{(\alpha_i,g_i,\beta_i)}\ni 0$, a contradiction. The obtained contradiction implies the statement of the proposition.
\end{proof}

\begin{corollary}\label{corollary-2.3}
Let $S$ be a Hausdorff primitive inverse semitopological semigroup and $S$ be an orthogonal sum of the family $\{B_{\lambda_{i}}(G_i)\colon i\in\mathscr{I}\}$ of semitopological Brandt semigroups with zeros. Then the following statements hold:
\begin{itemize}
    \item[$(i)$] for every non-zero element $(\alpha_i,g_i,\beta_i)\in (G_i)_{\alpha_i,\beta_i}\subseteq B_{\lambda_{i}}(G_i)\subseteq S$ there exists an open neighbourhood $U_{(\alpha_i,g_i,\beta_i)}$ of $(\alpha_i,g_i,\beta_i)$ in $S$ such that $U_{(\alpha_i,g_i,\beta_i)}\subseteq (G_i)_{\alpha_i,\beta_i}$ and hence every subset $(G_i)_{\alpha_i,\beta_i}$ is an open subset of $S$;

    \item[$(ii)$] every non-zero idempotent of $S$ is an isolated point of $E(S)$.
\end{itemize}
\end{corollary}

\begin{proof}
Assertion $(i)$ follows from Proposition~\ref{proposition-2.2} and $(ii)$ follows from $(i)$.
\end{proof}

\begin{proposition}\label{proposition-2.4}
Let $S$ be a Hausdorff countably compact semitopological semigroup such that $S$ is an orthogonal sum of the family $\{B_{\lambda_{i}}^0(S_i)\colon i\in\mathscr{I}\}$ of topological Brandt $\lambda^0_i$-extensions of semitopological monoids with zeros. Then for every open neighbourhood $U(0)$ of zero $0$ in $S$ the set of pairs of indices $(\alpha_i,\beta_i)$ such that $(S_i)_{\alpha_i,\beta_i}\nsubseteq U(0)$ is finite. Moreover, every maximal topological Brandt $\lambda^0_i$-extension $B_{\lambda_{i}}^0(S_i)$, $i\in\mathscr{I}$, is countably compact.
\end{proposition}

\begin{proof}
Suppose to the contrary that there exists an open neighbourhood
$U(0)$  of the zero $0$ in $S$ such that $(S_i)_{\alpha_i,\beta_i}\nsubseteq U(0)$ for infinitely many pairs of indices $(\alpha_i,\beta_i)$. Then for every
such $(S_i)_{\alpha_i,\beta_i}$ we choose a point $x_{{\alpha_i,\beta_i}}\in (S_i)_{\alpha_i,\beta_i}\setminus U(0)$ and put
$A=\bigcup\{x_{\alpha_i,\beta_i}\}$. Then $A$ is infinite and
Proposition~\ref{proposition-2.2} implies that the set $A$ has no accumulation
point in $S$. This contradicts Theorem~3.10.3 of \cite{Engelking1989}. The obtained
contradiction implies the first statement of the proposition.

The second statement follows from Proposition~\ref{proposition-2.2}, because by Theorem~3.10.4 of~\cite{Engelking1989} every closed subspace of a countably compact space is countably compact.
\end{proof}

Proposition~\ref{proposition-2.4} implies the following corollary:

\begin{corollary}\label{corollary-2.5}
Let $S$ be a Hausdorff primitive inverse countably compact semitopological semigroup and $S$ be an orthogonal sum of the family $\{B_{\lambda_{i}}(G_i)\colon i\in\mathscr{I}\}$ of semitopological Brandt semigroups with zeros. Then for every open neighbourhood $U(0)$ of zero $0$ in $S$ the set of pairs of indices $(\alpha_i,\beta_i)$ such that $(S_i)_{\alpha_i,\beta_i}\nsubseteq U(0)$ is finite. Moreover, every maximal Brandt subsemigroup $B_{\lambda_{i}}(G_i)$, $i\in\mathscr{I}$, is countably compact.
\end{corollary}

\begin{proposition}\label{proposition-2.6}
Let $S$ be a Hausdorff feebly compact semitopological semigroup such that $S$ is an orthogonal sum of the family $\{B_{\lambda_{i}}^0(S_i)\colon i\in\mathscr{I}\}$ of topological Brandt $\lambda^0_i$-extensions of semitopological monoids with zeros. Then
\begin{itemize}
  \item[$(i)$] every maximal topological Brandt $\lambda^0_i$-extension $B_{\lambda_{i}}^0(S_i)$, $i\in\mathscr{I}$, is feebly compact;
  \item[$(ii)$] the subspace $(S_i)_{\alpha_i,\beta_i}$ is feebly compact for all $\alpha_i,\beta_i\in \lambda_i$.
\end{itemize}
\end{proposition}

\begin{proof}
$(i)$ Let $\mathscr{F}=\{U_\alpha\colon \alpha\in\mathscr{J}\}$ be a infinite family of open non-empty subsets of $B_{\lambda_{i}}^0(S_i)$. If $0$ is contained in infinitely
many members of the family $\mathscr{F}$ then it is not locally finite.
In the opposite case the family $\mathscr{F}$ contains an infinite subfamily
$\mathscr{F}'$ no member of which contains $0$. Since the space $S$ is feebly compact, there exists a point $x\in S$ such that each neighbourhood of $x$ intersects
infinitely many members the family $\mathscr{F}'$. Suppose that $x\in U=
B_{\lambda_{j}}^0(S_j)\setminus\{0\}$ for some index $j\not=i$.
By Proposition~\ref{proposition-2.2}, $U$ is an open subset of $S$. But $U\cap U_\alpha=\varnothing$  for each member $U_\alpha$ of the family $\mathscr{F}'$.
Hence $x\in B_{\lambda_{i}}^0(S_i)$, a contradiction. Thus the family $\mathscr{F}'$ is not locally finite in $B_{\lambda_{i}}^0(S_i)$.

$(ii)$ Since the semigroup operation in $S$ is separately continuous the map $f_{\alpha_i,\beta_i}\colon S\rightarrow S\colon x\mapsto (\alpha_i,1_{S_i},\alpha_i)\cdot x\cdot(\beta_i,1_{S_i},\beta_i)$ is continuous too, and hence  $(S_i)_{\alpha_i,\beta_i}$ is a feebly compact subspace of $S$ as a continuous image of a feebly compact space.
\end{proof}

Proposition~\ref{proposition-2.6} implies the following corollary:

\begin{corollary}\label{corollary-2.7}
Let $S$ be a Hausdorff primitive inverse feebly compact semitopological semigroup and $S$ be an orthogonal sum of the family $\{B_{\lambda_{i}}(G_i)\colon i\in\mathscr{I}\}$ of semitopological Brandt semigroups with zeros. Then
\begin{itemize}
  \item[$(i)$] every maximal Brandt semigroup $B_{\lambda_{i}}(G_i)$, $i\in\mathscr{I}$, is feebly compact;
  \item[$(ii)$] $(G_i)_{\alpha_i,\beta_i}^0$ is feebly compact for all $\alpha_i,\beta_i\in \lambda_i$.
\end{itemize}
\end{corollary}

\begin{proposition}\label{proposition-2.8}
Let $S$ be a semiregular feebly compact semitopological semigroup such that $S$ is an orthogonal sum of the family $\{B_{\lambda_{i}}^0(S_i)\colon i\in\mathscr{I}\}$ of topological Brandt $\lambda^0_i$-extensions of semitopological monoids with zeros. Then for every open neighbourhood $U(0)$ of zero $0$ in $S$ the set of pairs of indices $(\alpha_i,\beta_i)$ such that $(S_i)_{\alpha_i,\beta_i}\nsubseteq U(0)$ is finite.
\end{proposition}

\begin{proof}
Since the semigroup $S$ is semiregular, there exists a regular open neighbourhood
$V(0)$ of the zero $0$ in $S$ such that $V(0)\subset U(0)$.
Let $\mathscr{A}=\left\{(\alpha_i,\beta_i)\colon (S_i)_{\alpha_i,\beta_i}\nsubseteq V(0)\right\}$. Let $(\alpha_i,\beta_i)\in \mathscr{A}$ be an arbitrary pair.
The set $(S_i)_{\alpha_i,\beta_i}'=(S_i)_{\alpha_i,\beta_i}^*\setminus
\operatorname{cl}_{S} V(0)$ is a non-empty open subset of the topological space $S$. Indeed, in the opposite case $(S_i)_{\alpha_i,\beta_i}\subseteq\operatorname{cl}_{S} V(0)$ and since by Proposition~\ref{proposition-2.2} the set  $(S_i)_{\alpha_i,\beta_i}^*$ is open and the set $V(0)$ is regular open,
we have $(S_i)_{\alpha_i,\beta_i}\subseteq \operatorname{int}_{S} (\operatorname{cl}_{S}(V(0)))=V(0)$, a contradiction. One can easily check that the family $\mathscr{P}=\{(S_i)_{\alpha_i,\beta_i}'\colon (\alpha_i,\beta_i)\in \mathscr{A}\}$ is a locally finite family of open subsets of the
topological space $S$. Since $S$ is feebly compact, the family $\mathscr{P}$ is finite, so the family $\mathscr{A}$ is finite too.
\end{proof}

Proposition~\ref{proposition-2.8} implies the following:

\begin{corollary}\label{corollary-2.9}
Let $S$ be a semiregular primitive inverse feebly compact semitopological semigroup and $S$ be an orthogonal sum of the family $\{B_{\lambda_{i}}(G_i)\colon i\in\mathscr{I}\}$ of semitopological Brandt semigroups with zeros. Then for every open neighbourhood $U(0)$ of zero $0$ in $S$ the set of pairs of indices $(\alpha_i,\beta_i)$ such that $(G_i)_{\alpha_i,\beta_i}\nsubseteq U(0)$ is finite.
\end{corollary}

The structure of primitive Hausdorff feebly compact topological inverse semigroup is described in \cite{GutikPavlyk2013??}. It is proved that every primitive Hausdorff feebly compact topological inverse semigroup $S$ is topologically isomorphic to the orthogonal sum $\sum_{i\in\mathscr{I}}B_{\lambda_{i}}(G_i)$ of topological Brandt $\lambda_i$-extensions $B_{\lambda_i}(G_i)$ of pseudocompact topological groups $G_i$ in the class of topological inverse semigroups for some finite cardinals $\lambda_i\geqslant 1$. Also, \cite{GutikPavlyk2013??} contains a description of a base of the topology of a primitive Hausdorff feebly compact topological inverse semigroup. Similar results for primitive Hausdorff countably compact topological inverse semigroups and Hausdorff compact topological inverse semigroups were obtained in \cite{BerezovskiGutikPavlyk2010}.

The following example shows that counterparts of these results do not hold for  primitive Hausdorff compact (and hence countably compact and feebly compact) semitopological inverse semigroups with continuous inversion.

\begin{example}\label{example-2.10}
Let $\mathbb{Z}(+)$ be the discrete additive group of integers and $\mathcal{O}\notin\mathbb{Z}(+)$. We put $Z^0$ to be $\mathbb{Z}(+)$ with adjoined zero $\mathcal{O}$ and consider the topology of the one-point Alexandroff compactification on $Z^0$ with the remainder $\{\mathcal{O}\}$. Simple verifications show that $Z^0$ is a Hausdorff compact semitopological inverse semigroup with continuous inversion.

We fix an arbitrary cardinal $\lambda\geqslant 1$. Define a topology $\tau_B$ on $B^0_{\lambda}(Z^0)$ as follows:
\begin{itemize}
  \item[$(i)$] all non-zero elements of $B^0_{\lambda}(Z^0)$ are isolated points;
  \item[$(ii)$] the family $\mathscr{P}(0)=\{U(\alpha,\beta,n)\colon \alpha,\beta\in\lambda, \; n\in\mathbb{N}\}$, where
\begin{equation*}
U(\alpha,\beta,n)=B^0_{\lambda}(Z^0)\setminus (\{-n,-n+1,\ldots,n-1,n\})_{\alpha,\beta},
\end{equation*}
forms a pseudobase of the topology $\tau_B$ at zero.
\end{itemize}
Simple verifications show that $(B^0_{\lambda}(Z^0),\tau_B)$ is a Hausdorff compact semitopological inverse semigroup with continuous inversion, and moreover the space $(B^0_{\lambda}(Z^0),\tau_B)$ is homeomorphic to the one-point Alexandroff compactification of the discrete space of cardinality $\max\{\lambda,\omega\}$ with the remainder zero of the semigroup $B^0_{\lambda}(Z^0)$.
\end{example}

\begin{theorem}\label{theorem-2.11}
Let $S$ be a Hausdorff primitive inverse countably compact semitopological semigroup and $S$ be an orthogonal sum of the family $\{B_{\lambda_{i}}(G_i)\colon i\in\mathscr{I}\}$ of semitopological Brandt semigroups with zeros. Suppose that for every $i\in\mathscr{I}$ there exists a maximal non-zero subgroup $(G_i)_{\alpha_i,\alpha_i}$, $\alpha_i\in\lambda_i$, such that at least the one of the following conditions holds:
\begin{itemize}
  \item[(1)] the group $(G_i)_{\alpha_i,\alpha_i}$ is left precompact;
  \item[(2)] $(G_i)_{\alpha_i,\alpha_i}$ is a feebly compact paratopological group;
  \item[(3)] the group $(G_i)_{\alpha_i,\alpha_i}$ is left $\omega$-precompact feebly compact;
  \item[(4)] the semigroup $S_{\alpha_i,\alpha_i}= (G_i)_{\alpha_i,\alpha_i}\cup\{0\}$ is a topological semigroup.
\end{itemize}
Then the following assertions hold:
\begin{enumerate}
  \item[$(i)$] every maximal subgroup of $S$ is a closed subset of $S$ and hence is countably compact;
  \item[$(ii)$] for every $i\in\mathscr{I}$ the maximal Brandt semigroup $B_{\lambda_{i}}(G_i)$ is a countably compact topological Brandt $\lambda$-extension of a countably compact semitopological group $G_i$;
  \item[$(iii)$] if $\mathscr{B}_{(\alpha_i,e_{i},\alpha_i)}$ is a base of the topology at the unit $(\alpha_i,e_{i},\alpha_i)$ of a maximal non-zero subgroup $(G_i)_{\alpha_i,\alpha_i}$ of $S$, $i\in\mathscr{I}$, such that $U\subseteq (G_i)_{\alpha_i,\alpha_i}$ for any $U\in \mathscr{B}_{(\alpha_i,e_{i},\alpha_i)}$, then the family
      \begin{equation*}
        \mathscr{B}_{(\beta_i,x,\gamma_i)}=\left\{(\beta_i,x,\alpha_i)\cdot U\cdot (\alpha_i,e_i,\gamma_i)\colon U\in \mathscr{B}_{(\alpha_i,e_{i},\alpha_i)} \right\}
      \end{equation*}
      is a base of the topology of $S$ at the point $(\beta_i,x,\gamma_i)\in (G_i)_{\beta_i,\gamma_i}\subseteq B_{\lambda_{i}}(G_i)$, for all $\beta_i,\gamma_i\in \lambda_i$;
  \item[$(iv)$] the family
      \begin{equation*}
      \begin{split}
        \mathscr{B}_{0}= \big\{S\setminus& \big((G_{i_1})_{\alpha_{i_1},\beta_{i_1}}\cup \cdots\cup (G_{i_k})_{\alpha_{i_k},\beta_{i_k}}\big)\colon i_1,\ldots,i_k\in\mathscr{I}, \alpha_{i_k},\beta_{i_k}\in\lambda_{i_k},\\
          & k\in\mathbb{N},  \{(\alpha_{i_1},\beta_{i_1}),\ldots,(\alpha_{i_k},\beta_{i_k})\} \hbox{~is finite}\big\}
      \end{split}
      \end{equation*}
      is a base of the topology at zero of $S$.
\end{enumerate}
\end{theorem}

\begin{proof}
$(i)$ Fix an arbitrary maximal subgroup $G$ of $S$. Without loss of generality we can assume that $G$ is a non-zero subgroup of $S$. Then there exists a maximal Brand subsemigroup $B_{\lambda_{i}}(G_i)$, $i\in\mathscr{I}$, which contains $G$. The separate continuity of multiplication in $S$ implies that for all $\alpha_i,\beta_i,\gamma_i,\delta_i\in\lambda_i$ the map $\psi^{\alpha_i,\beta_i}_{\gamma_i,\delta_i}\colon S\to S$ defined by the formula $\psi^{\alpha_i,\beta_i}_{\gamma_i,\delta_i}(x)=(\gamma_i,e_i,\alpha_i)\cdot x\cdot(\beta_i,e_i,\delta_i)$, where $e_i$ is unit of the group $G_i$, is continuous. Since for all $\alpha_i,\beta_i,\gamma_i,\delta_i\in\lambda_i$ the restrictions $\psi^{\alpha_i,\beta_i}_{\gamma_i,\delta_i}|_{(G_i)_{\alpha_i,\beta_i}}\colon (G_i)_{\alpha_i,\beta_i}\to (G_i)_{\gamma_i,\delta_i}$ and $\psi^{\gamma_i,\delta_i}_{\alpha_i,\beta_i}|_{(G_i)_{\gamma_i,\delta_i}}\colon (G_i)_{\gamma_i,\delta_i}\to (G_i)_{\alpha_i,\beta_i}$ are bijective continuous maps we conclude that $(G_i)_{\alpha_i,\beta_i}$ and $(G_i)_{\gamma_i,\delta_i}$ are homeomorphic subspaces of $S$, and moreover the semitopological subgroups $(G_i)_{\alpha_i,\alpha_i}$ and $(G_i)_{\gamma_i,\gamma_i}$ are topologically isomorphic for all indices $\alpha_i,\gamma_i\in\lambda_i$. Therefore $G$ is topologically isomorphic to the semitopological subgroup $(G_i)_{\alpha_i,\alpha_i}$ for any $\alpha_i\in\lambda_i$. For any $\alpha_i,\beta_i\in\lambda_i$ we put $S_{\alpha_i,\beta_i}=(G_i)_{\alpha_i,\beta_i} \cup\{0\}$. Then for all $\alpha_i,\beta_i,\gamma_i,\delta_i\in\lambda_i$ the restrictions $\psi^{\alpha_i,\beta_i}_{\gamma_i,\delta_i}|_{S_{\alpha_i,\beta_i}}\colon S_{\alpha_i,\beta_i}\to S_{\gamma_i,\delta_i}$ and $\psi^{\gamma_i,\delta_i}_{\alpha_i,\beta_i}|_{S_{\gamma_i,\delta_i}}\colon S_{\gamma_i,\delta_i}\to S_{\alpha_i,\beta_i}$ are bijective continuous maps, and hence $S_{\alpha_i,\beta_i}$ and $S_{\gamma_i,\delta_i}$ are homeomorphic subspaces of $S$, and moreover the semitopological subsemigroups $S_{\alpha_i,\alpha_i}$ and $S_{\gamma_i,\gamma_i}$ are topologically isomorphic for all indices $\alpha_i,\gamma_i,\in\lambda_i$.  Now Lemma~\ref{lemma-2.1} implies that $0$ is an isolated point in $S_{\alpha_i,\alpha_i}$. Indeed, if one of Conditions (1)-(3) of the theorem is satisfied then we can directly apply Lemma~\ref{lemma-2.1} and if Condition (4) of the theorem is satisfied then we observe that for each $\lambda_i$ and $\alpha_i\in\lambda_i$ the subsemigroup $S_{\alpha_i,\alpha_i}$ of $S$ is countably compact as a retract of $S$ and hence $S_{\alpha_i,\alpha_i}$ is feebly compact and  then again Lemma~\ref{lemma-2.1} applies. By Corollary~\ref{corollary-2.3}, $(G_i)_{\alpha_i,\alpha_i}$ is a closed subspace of $S$ and by Theorem~3.10.4 from \cite{Engelking1989} $(G_i)_{\alpha_i,\alpha_i}$ is countably compact, and hence so is $G$, too.

$(ii)$ The arguments presented in the proof of the assertion $(i)$ imply that for every $i\in\mathscr{I}$ the maximal Brandt semigroup $B_{\lambda_{i}}(G_i)$ is a topological Brandt $\lambda$-extension of a countably compact semitopological group $G_i$. By Corollary~\ref{corollary-2.3} we have that for every $i\in\mathscr{I}$ the maximal Brandt semigroup $B_{\lambda_{i}}(G_i)$ is a closed subset of $S$ and by Theorem~3.10.4 from \cite{Engelking1989} $B_{\lambda_{i}}(G_i)$ is countably compact.

Assertion $(iii)$ follows from $(ii)$.

$(iv)$ follows from Corollary~\ref{corollary-2.5} and assertions $(i)$ and $(ii)$.
\end{proof}

The proof of the following theorem is similar to the proof of Theorem~\ref{theorem-2.11} and makes use of Corollary~\ref{corollary-2.7} and Proposition~\ref{proposition-2.8}.

\begin{theorem}\label{theorem-2.12}
Let $S$ be a Hausdorff primitive inverse feebly compact semitopological semigroup and $S$ be an orthogonal sum of the family $\{B_{\lambda_{i}}(G_i)\colon i\in\mathscr{I}\}$ of semitopological Brandt semigroups with zeros. Suppose that for every $i\in\mathscr{I}$ there exists a maximal non-zero subgroup $(G_i)_{\alpha_i,\alpha_i}$, $\alpha_i\in\lambda_i$, such that at least the one of the following conditions holds:
\begin{itemize}
  \item[(1)] the group $(G_i)_{\alpha_i,\alpha_i}$ is left precompact;
  \item[(2)] $(G_i)_{\alpha_i,\alpha_i}$ is a feebly compact paratopological group;
  \item[(3)] the group $(G_i)_{\alpha_i,\alpha_i}$ is left $\omega$-precompact feebly compact;
  \item[(4)] the semigroup $S_{\alpha_i,\alpha_i}= (G_i)_{\alpha_i,\alpha_i}\cup\{0\}$ is a topological semigroup.
\end{itemize}
Then the following assertions hold:
\begin{enumerate}
  \item[$(i)$] every maximal subgroup of $S$ is an open-and-closed subset of $S$ and hence is pseudocompact;
  \item[$(ii)$] for every $i\in\mathscr{I}$ the maximal Brandt semigroup $B_{\lambda_{i}}(G_i)$ is a feebly compact topological Brandt $\lambda$-extension of a feebly compact semitopological group $G_i$;
  \item[$(iii)$] if $\mathscr{B}_{(\alpha_i,e_{i},\alpha_i)}$ is a base of the topology at the unit $(\alpha_i,e_{i},\alpha_i)$ of a maximal non-zero subgroup $(G_i)_{\alpha_i,\alpha_i}$ of $S$, $i\in\mathscr{I}$, such that $U\subseteq (G_i)_{\alpha_i,\alpha_i}$ for any $U\in \mathscr{B}_{(\alpha_i,e_{i},\alpha_i)}$, then the family
      \begin{equation*}
        \mathscr{B}_{(\beta_i,x,\gamma_i)}=\left\{(\beta_i,x,\alpha_i)\cdot U\cdot (\alpha_i,e_i,\gamma_i)\colon U\in \mathscr{B}_{(\alpha_i,e_{i},\alpha_i)} \right\}
      \end{equation*}
      is a base of the topology at the point $(\beta_i,x,\gamma_i)\in (G_i)_{\beta_i,\gamma_i}\subseteq B_{\lambda_{i}}(G_i)$, for all $\beta_i,\gamma_i\in \lambda_i$;
\end{enumerate}
if in addition the topological space $S$ is semiregular then
\begin{enumerate}
  \item[$(iv)$] the family
      \begin{equation}\label{eq-2.1}
      \begin{split}
        \mathscr{B}_{0}= \big\{S\setminus& \big((G_{i_1})_{\alpha_{i_1},\beta_{i_1}}\cup \cdots\cup (G_{i_k})_{\alpha_{i_k},\beta_{i_k}}\big)\colon i_1,\ldots,i_k\in\mathscr{I}, \alpha_{i_k},\beta_{i_k}\in\lambda_{i_k}, k\in\mathbb{N}; \\
          & \hbox{if~} \left\{(\alpha_{i_1},\beta_{i_1}),\ldots,(\alpha_{i_k},\beta_{i_k})\right\} \hbox{~is finite}\big\}
      \end{split}
      \end{equation}
      is a base of the topology at zero of $S$.
\end{enumerate}
\end{theorem}

The following example shows that in the case of primitive Hausdorff feebly compact semitopological inverse semigroups with compact maximal subgroups and continuous inversion the statement $(iii)$ of Theorem~\ref{theorem-2.12} doesn't hold.

\begin{example}\label{example-2.13}
Let $\lambda$ be an infinite cardinal and $\mathbb{T}$ be the unit circle with the usual multiplication of complex numbers and the usual topology $\tau_{\mathbb{T}}$. It is obvious that $(\mathbb{T},\tau_{\mathbb{T}})$ is a topological group. The base of the topology $\tau_B$ on the Brandt semigroup $B_\lambda(\mathbb{T})$ we define as follows:
\begin{enumerate}
  \item[1)] for every non-zero element $(\alpha,x,\beta)$ of the semigroup $B_\lambda(\mathbb{T})$ the family
      \begin{equation*}
        \mathscr{B}_{(\alpha,x,\beta)}=\left\{(\alpha,U(x),\beta)\colon U(x)\in \mathscr{B}_\mathbb{T}(x)\right\},
      \end{equation*}
      where $\mathscr{B}_\mathbb{T}(x)$ is a base of the topology $\tau_\mathbb{T}$ at the point $x\in \mathbb{T}$, is the base of the topology $\tau_{\mathbb{T}}$ at $(\alpha,x,\beta)\in B_\lambda(\mathbb{T})$;
  \item[2)] the family
      \begin{equation*}
        \mathscr{B}_{0}=\left\{U(\alpha_1,\beta_1;\ldots;\alpha_n,\beta_n; x_1,\ldots,x_k)\colon \alpha_1,\beta_1,\ldots,\alpha_n,\beta_n\in\lambda, x_1,\ldots,x_k\in\mathbb{T}, n,k\in\mathbb{N}\right\},
      \end{equation*}
      where
      \begin{equation*}
        U(\alpha_1,\beta_1;\ldots;\alpha_n,\beta_n; x_1,\ldots,x_k)= B_\lambda(\mathbb{T})\setminus\left(\mathbb{T}_{\alpha_1,\beta_1} \cup\cdots\cup\mathbb{T}_{\alpha_n,\beta_n}\cup\left\{(\alpha,x_i,\beta)\colon \alpha,\beta\in\lambda, i=1,\ldots,k\right\}\right),
      \end{equation*}
      is the base of the topology $\tau_{\mathbb{T}}$ at zero $0\in B_\lambda(\mathbb{T})$.
\end{enumerate}

Simple verifications show that $(B_\lambda(\mathbb{T}),\tau_B)$ is a non-semiregular Hausdorff feebly compact topological space for every infinite cardinal $\lambda$. Next we shall show that multiplication on $(B_\lambda(\mathbb{T}),\tau_B)$ is separately continuous. The proof of the separate continuity of multiplication in the cases $0\cdot 0$ and $(\alpha,x,\beta)\cdot(\gamma,y,\delta)$, where $\alpha,\beta,\gamma,\delta\in\lambda$ and $x,y\in \mathbb{T}$, is trivial, and hence we only consider the following cases:
\begin{equation*}
     (\alpha,x,\beta)\cdot 0 \qquad \hbox{and} \qquad 0\cdot(\alpha,x,\beta).
\end{equation*}

Then we have that
\begin{equation*}
\begin{split}
  (\alpha,x,\beta)\cdot &  U(\beta,\beta_1;\ldots;\beta,\beta_n;\alpha_1,\beta_1;\ldots;\alpha_n,\beta_n; x_1,\ldots,x_k)\subseteq \\
    \subseteq &\;\{0\}\cup \bigcup\left\{\mathbb{T}_{\alpha,\gamma}\setminus \left\{(\alpha,xx_1,\gamma),\ldots,(\alpha,xx_k,\gamma)\right\}\colon \gamma\in\lambda\setminus\{\beta_1,\ldots,\beta_n\}\right\}\subseteq \\
    \subseteq &\; U(\alpha,\beta_1;\ldots;\alpha,\beta_n;\alpha_1,\beta_1;\ldots;\alpha_n,\beta_n; xx_1,\ldots,xx_k)\subseteq \\
    \subseteq &\; U(\alpha_1,\beta_1;\ldots;\alpha_n,\beta_n; xx_1,\ldots,xx_k)
\end{split}
\end{equation*}
and similarly
\begin{equation*}
\begin{split}
  U&(\alpha_1,\alpha;\ldots;\alpha_n,\alpha;\alpha_1,\beta_1;\ldots;\alpha_n,\beta_n; x_1,\ldots,x_k)\cdot (\alpha,x,\beta)\subseteq \\
    &\; \subseteq \{0\}\cup \bigcup\left\{\mathbb{T}_{\gamma,\beta}\setminus \left\{(\gamma,x_1x,\beta),\ldots,(\gamma,x_kx,\beta)\right\}\colon \gamma\in\lambda\setminus\{\alpha_1,\ldots,\alpha_n\}\right\}\subseteq \\
    &\; \subseteq U(\alpha_1,\beta;\ldots;\alpha_n,\beta;\alpha_1,\beta_1;\ldots;\alpha_n,\beta_n; x_1 x,\ldots,x_k x)\subseteq \\
    &\; \subseteq U(\alpha_1,\beta_1;\ldots;\alpha_n,\beta_n; x_1 x,\ldots,x_k x),
\end{split}
\end{equation*}
for all $U(\alpha_1,\beta_1;\ldots;\alpha_n,\beta_n; xx_1,\ldots,xx_k), U(\alpha_1,\beta_1;\ldots;\alpha_n,\beta_n; x_1 x,\ldots,x_k x)\in\mathscr{B}_{0}$. This completes the proof of the separate continuity of multiplication in $(B_\lambda(\mathbb{T}),\tau_B)$.
\end{example}

\begin{proposition}\label{proposition-2.14}
The space $(B_\lambda(\mathbb{T}),\tau_B)$ is countably pracompact if and only if $\lambda\leqslant\mathfrak{c}$.
\end{proposition}

\begin{proof}
$(\Leftarrow)$ Suppose that $\lambda\leqslant\mathfrak{c}$. Then there exists a countable dense subgroup $H$ of $\mathbb{T}$. Let $\mathfrak{H}_H$ be the family of all distinct conjugate classes of subgroup $H$ in $\mathbb{T}$. Since the subgroup $H$ is countable we conclude that the cardinality of $\mathfrak{H}_H$ is $\mathfrak{c}$. This implies that there exists a one-to-one (not necessary bijective) map $f\colon \lambda\times\lambda\to \mathfrak{H}_H\colon (\alpha,\beta)\mapsto g_{\alpha,\beta}H$. Then by the definition of the topology $\tau_B$ we have that $A=\bigcup_{\alpha,\beta\in\lambda}(g_{\alpha,\beta}H)_{\alpha,\beta}$ is a dense subset of the topological space $(B_\lambda(\mathbb{T}),\tau_B)$. Fix an arbitrary infinite countable subset $Q$ of $A$. If the set $Q\cap\mathbb{T}_{\alpha,\beta}$ is infinite for some $\alpha,\beta\in\lambda$ then compactness of $\mathbb{T}$ implies that $Q$ has an accumulation point in $\mathbb{T}_{\alpha,\beta}$, and hence in $(B_\lambda(\mathbb{T}),\tau_B)$. In the other case the definition of the topology $\tau_B$ implies that zero $0$ is an accumulation point of $Q$. Therefore the space $(B_\lambda(\mathbb{T}),\tau_B)$ is countably compact at $A$, and hence it is countably pracompact.

$(\Rightarrow)$ Suppose that there exists a cardinal $\lambda>\mathfrak{c}$ such that the space $(B_\lambda(\mathbb{T}),\tau_B)$ is countably pracompact. Then there exists a dense subset $A$ of $(B_\lambda(\mathbb{T}),\tau_B)$ such that the space $(B_\lambda(\mathbb{T}),\tau_B)$ is countably compact at $A$. The definition of the topology $\tau_B$ implies that $A\cap\mathbb{T}_{\alpha,\beta}$ is dense subset in $\mathbb{T}_{\alpha,\beta}$ for all $\alpha,\beta\in\lambda$. Since $\lambda>\mathfrak{c}$ and $|\mathbb{T}|=\mathfrak{c}$ we conclude that there exists a point $x\in\mathbb{T}$ such that $(\alpha,x,\beta)\in A$ for infinitely many distinct pairs $(\alpha,\beta)$ of indices in $\lambda$. Put $\mathcal{K}=\{(\alpha,\beta)\in\lambda\times\lambda\colon (\alpha,x,\beta)\in A\}$.
The definition of the topology $\tau_B$ implies that for every infinite countable subset $\mathcal{K}_0\subseteq \mathcal{K}$ the set $\{(\alpha,x,\beta)\colon (\alpha,\beta)\in\mathcal{K}_0\}$ has no accumulation point in $(B_\lambda(\mathbb{T}),\tau_B)$, a contradiction.
\end{proof}

The proof of the following proposition is similar to the proof of Proposition~22 from \cite{Gutik-Pavlyk-2013}.

\begin{proposition}\label{proposition-2.20}
Let $S$ be a semiregular feebly compact (Hausdorff countably compact) semitopological semigroup such that $S$ is an orthogonal sum of the family $\{B_{\lambda_{i}}^0(S_i)\colon i\in\mathscr{I}\}$ of topological Brandt $\lambda^0_i$-extensions of semitopological monoids with zeros, i.e. $S=\sum_{i\in\mathscr{I}}B_{\lambda_{i}}^0(S_i)$. Then the following assertions hold:
\begin{itemize}
  \item[$(i)$] the topological space $S$ is regular if and only if the space $S_i$ is regular for each $i\in\mathscr{I}$;

  \item[$(ii)$] the topological space $S$ is Tychonoff if and only if the space $S_i$ is Tychonoff  for each $i\in\mathscr{I}$;

  \item[$(iii)$] the topological space $S$ is normal if and only if the space $S_i$ is normal for each $i\in\mathscr{I}$.
\end{itemize}
\end{proposition}

%%%%%%%%%%%%%%%%%%%%%%%%%%%%%%%%%%%%%%%%%%%%%%%%%%%%%

The following theorem characterizes feebly compact topological Brandt $\lambda^0$-extensions of topological monoids with zero in the class of Hausdorff topological semigroups.

\begin{theorem}\label{theorem-3.1}
A topological Brandt $\lambda^0$-extension $\left(B^0_{\lambda}(S), \tau_{B}\right)$ of a topological monoid $(S,\tau_S)$ with zero in the class of Hausdorff topological semigroups is feebly compact if and only if the cardinal $\lambda$ is finite and the space $(S,\tau_S)$ is feebly compact.
\end{theorem}

\begin{proof}
$(\Leftarrow)$ The continuity of multiplication in $\left(B^0_{\lambda}(S), \tau_{B}\right)$ implies that for all $\alpha,\beta,\gamma,\delta\in\lambda$ the map $\psi^{\alpha,\beta}_{\gamma,\delta}\colon B^0_{\lambda}(S)\to B^0_{\lambda}(S)$ defined by the formula $\psi^{\alpha,\beta}_{\gamma,\delta}(x)= (\gamma,1_S,\alpha)\cdot x\cdot(\beta,1_S,\delta)$, where $1_S$ is unit of the semigroup $S$, is continuous. Since for all $\alpha,\beta,\gamma,\delta\in\lambda$ the restrictions $\psi^{\alpha,\beta}_{\gamma,\delta}|_{S_{\alpha,\beta}}\colon S_{\alpha,\beta}\to S_{\gamma,\delta}$ and $\psi^{\gamma,\delta}_{\alpha,\beta}|_{S_{\gamma,\delta}}\colon S_{\gamma,\delta}\to S_{\alpha,\beta}$ are bijective continuous maps we conclude that $S_{\alpha,\beta}$ and $S_{\gamma,\delta}$ are homeomorphic subspaces of $\left(B^0_{\lambda}(S), \tau_{B}\right)$. Therefore the space $\left(B^0_{\lambda}(S), \tau_{B}\right)$ is the union of finitely many copies of the feebly compact topological space $(S,\tau_S)$, and hence it is feebly compact.

$(\Rightarrow)$  Suppose that a topological Brandt $\lambda^0$-extension $\left(B^0_{\lambda}(S), \tau_{B}\right)$ of a topological monoid $(S,\tau_S)$ with zero in the class of topological semigroups is feebly compact. Then by Proposition~\ref{proposition-2.6}$(ii)$ the space $(S,\tau_S)$ is feebly compact.

Suppose for a contradiction that there exists a feebly compact topological Brandt $\lambda^0$-extension $\left(B^0_{\lambda}(S), \tau_{B}\right)$ of a topological monoid $(S,\tau_S)$ with zero in the class of Hausdorff topological semigroups such that the cardinal $\lambda$ is infinite. Then the Hausdorffness of $\left(B^0_{\lambda}(S), \tau_{B}\right)$ implies that for every $\alpha\in\lambda$ there exist open disjoint neighbourhoods $U_0$ and $U_{(\alpha,1_S,\alpha)}$ of zero and $(\alpha,1_S,\alpha)$ in $\left(B^0_{\lambda}(S), \tau_{B}\right)$, respectively. Without loss of generality we can assume that $U_{(\alpha,1_S,\alpha)}=(U(1_S))_{\alpha,\alpha}$ for some open neighbourhood $U(1_S)$ of unit $1_S$ in $(S,\tau_S)$ (see Proposition~\ref{proposition-2.2}). By the continuity of multiplication in $\left(B^0_{\lambda}(S), \tau_{B}\right)$ there exists an open neighbourhood $V_0$ of zero in $\left(B^0_{\lambda}(S), \tau_{B}\right)$ such that $V_0\cdot V_0\subseteq U_0$. Also the continuity of multiplication in $(S,\tau_S)$ implies that there exists an open neighbourhood $V(1_S)$ of unit $1_S$ in $(S,\tau_S)$ such that $V(1_S)\cdot V(1_S)\subseteq U(1_S)$ in $S$.

Then the feeble compactness of $\left(B^0_{\lambda}(S), \tau_{B}\right)$ implies that zero $0$ is an accumulation point of each infinite subfamily of $\left\{(V(1_S))_{\alpha,\beta}\colon \alpha,\beta\in\lambda\right\}$ and hence $V_0\cap (V(1_S))_{\alpha,\beta}=\varnothing$ only for finitely many pairs of indices $(\alpha,\beta)$. Hence by the definition of multiplication on $B^0_{\lambda}(S)$ we have that $(V_0\cdot V_0)\cap U_{(\alpha,1_S,\alpha)}\neq\varnothing$. This contradicts the assumption that $U_0\cap U_{(\alpha,1_S,\alpha)}=\varnothing$. The obtained contradiction implies that cardinal $\lambda$ is finite.
\end{proof}

Theorem~\ref{theorem-3.1} implies the following corollary:

\begin{corollary}\label{corollary-3.1a}
A feebly compact topological Brandt $\lambda^0$-extension  of a topological inverse monoid with zero in the class of Hausdorff topological semigroups is a topological inverse semigroup.
\end{corollary}

The following example shows that there exists a compact topological semigroup with a non-pseudo\-com\-pact topological Brandt $2^0$-extension in the class of topological semigroups and hence the counterpart of Theorem~\ref{theorem-3.1} does not necessarily hold for semigroups without non-zero idempotent.

\begin{example}\label{example-3.2}
Let $X$ be any infinite Hausdorff compact topological space. Fix an arbitrary $z\in X$ and define multiplication on $X$ in the following way: $x\cdot y=z$ for all $x,y\in X$. It is obvious that this operation is continuous on $X$ and $z$ is zero of $X$. The set $X$ endowed with such an operation is called a \emph{semigroup with zero-multiplication}. We define the topology $\tau_B$ on the Brandt $2^0$-extension $B_2^0(X)$ of the semigroup $X$ as follows:
\begin{itemize}
  \item[$(i)$] the family $\mathscr{B}(0)=\left\{U_{1,1}\cup U_{2,2}\colon U\in\mathscr{B}(z)\right\}$, where $\mathscr{B}(z)$ is a base of the topology of $X$ at $z$, is the base of topology $\tau_B$ at zero of $B_2^0(X)$;
  \item[$(ii)$] for $i=1,2$ and any $x\in X\setminus\{z\}$ the family $\mathscr{B}_{(i,x,i)}=\left\{U_{i,i}\colon U\in\mathscr{B}(x)\right\}$, where $\mathscr{B}(x)$ is a base of the topology of $X$ at the point $x$, is the base of topology $\tau_B$ at the point $(i,x,i)\in B_2^0(X)$;
  \item[$(iii)$] all points of the subsets $X_{1,2}^*$ and $X_{2,1}^*$ are isolated points in $(B_2^0(X),\tau_B)$.
\end{itemize}
It is obvious that $B_2^0(X)$ is a semigroup with zero-multiplication. Simple verifications show that $\tau_B$ is a Hausdorff topology on $B_2^0(X)$. Hence $(B_2^0(X),\tau_B)$ is a topological semigroup and $(B_2^0(X),\tau_B)$ is a topological Brandt $2^0$-extension of $X$ in the class of topological semigroups. Since $X_{1,2}^*$ and $X_{2,1}^*$ are discrete open-and-closed subspaces of $(B_2^0(X),\tau_B)$ we have that the topological  space $(B_2^0(X),\tau_B)$ is not feebly compact.
\end{example}

Also, the following example shows that there exists a compact topological semigroup $S$ such that for every infinite cardinal $\lambda$ there exists a compact (and hence feebly compact) topological Brandt $\lambda^0$-extension $B_\lambda^0(S)$ of the semigroup $S$ in the class of topological semigroups.

\begin{example}\label{example-3.3}
Let $X$ be a compact topological semigroup defined in Example~\ref{example-3.2} and $\lambda$ be an arbitrary infinite cardinal. We define the topology $\tau_B$ on the Brandt $\lambda^0$-extension $B_\lambda^0(X)$ of the semigroup $X$ as follows:
\begin{itemize}
  \item[$(i)$] the family $\mathscr{B}_B(0)=\big\{U_{A}(0)= \bigcup_{(\alpha,\beta)\in(\lambda\times\lambda) \setminus A} X_{\alpha,\beta}\cup \bigcup_{(\gamma,\delta)\in A} (U(z))_{\gamma,\delta} \colon A \hbox{~is a finite subset of~} \lambda\times\lambda$ and  $U(z)\in\mathscr{B}_X(z)\big\}$, where $\mathscr{B}_X(z)$ is a base of the topology $x\in X$, is a base of topology $\tau_B$ at zero of $B_\lambda^0(X)$;
  \item[$(ii)$] for all $\alpha,\beta\in\lambda$ and any $x\in X\setminus \{z\}$ the family $\mathscr{B}_{(\alpha,x,\beta)}=\left\{U_{\alpha,\beta}\colon U\in\mathscr{B}(x)\right\}$, where $\mathscr{B}_X(x)$ is a base of the topology of $X$ at the point $x$, is the base of topology $\tau_B$ at the point $(\alpha,x,\beta)\in B_\lambda^0(X)$.
\end{itemize}
It is obvious that $B_\lambda^0(X)$ is a semigroup with zero-multiplication. Simple verifications show that $\tau_B$ is a Hausdorff compact topology on $B_\lambda^0(X)$. Hence $(B_\lambda^0(X),\tau_B)$ is a topological semigroup and $(B_\lambda^0(X),\tau_B)$ is a compact topological Brandt $\lambda^0$-extension of $X$ in the class of topological semigroups.
\end{example}

The following proposition extends Theorem~\ref{theorem-3.1}.

\begin{proposition}\label{proposition-3.4}
Let $S$ be a Hausdorff feebly compact topological semigroup such that $S$ is an orthogonal sum of the family $\{B_{\lambda_{i}}^0(S_i)\colon i\in\mathscr{I}\}$ of topological Brandt $\lambda^0_i$-extensions of topological semigroups with zeros, i.e. $S=\sum_{i\in\mathscr{I}}B_{\lambda_{i}}^0(S_i)$. If for some $i\in\mathscr{I}$ the semigroup $S_i$ has a non-zero idempotent then cardinal $\lambda_i$ is finite.
\end{proposition}

\begin{proof}
Suppose for a contradiction that there exists $i\in\mathscr{I}$ such that the cardinal $\lambda_i$ is infinite. Let $e$ be a non-zero idempotent of $S_i$. Then the Hausdorffness of $S$ implies that for every $\alpha_i\in\lambda_i$ there exist open disjoint neighbourhoods $U_0$ and $U_{(\alpha_i,e,\alpha_i)}$ of zero and $(\alpha_i,e,\alpha_i)$ in $S$, respectively. By continuity of multiplication in $S$ there exists an open neighbourhood $V_{(\alpha_i,e,\alpha_i)}$ of $(\alpha_i,e,\alpha_i)$ in $S$ such that $(\alpha_i,e,\alpha_i)\cdot V_{(\alpha_i,e,\alpha_i)}\cdot (\alpha_i,e,\alpha_i)\subseteq U_{(\alpha_i,e,\alpha_i)}$. This implies that $V_{(\alpha_i,e,\alpha_i)}\subseteq (S^*_i)_{\alpha_i,\alpha_i}$. Therefore without loss of generality we can assume that $U_{(\alpha_i,e,\alpha_i)}=(U(e))_{\alpha_i,\beta_i}$ for some open neighbourhood $U(e)$ of the idempotent $e$ in $S_i$. By continuity of multiplication in $S$ there exists an open neighbourhood $V_0$ of zero in $S$ such that $V_0\cdot V_0\subseteq U_0$. Also the continuity of the semigroup operation in $S_i$ implies that there exists an open neighbourhood $V(e)$ of the idempotent $e$ in $S_i$ such that $V(e)\cdot V(e)\subseteq U(e)$ in $S_i$.

Then the feeble compactness of $S$ implies that zero $0$ is an accumulation point of each infinite subfamily of $\left\{(V(1_S))_{\alpha_i,\beta_i}\colon \alpha_i,\beta_i\in\lambda_i, i\in\mathscr{I}\right\}$ and hence $V_0\cap (V(1_S))_{\alpha_i,\beta_i}=\varnothing$ only for finitely many pairs if indices $(\alpha_i,\beta_i)$ from $\lambda_i$, $i\in\mathscr{I}$. Hence by the definition of multiplication on $S$ we have that $(V_0\cdot V_0)\cap U_{(\alpha_i,e,\alpha_i)}\neq\varnothing$. This contradicts the assumption $U_0\cap U_{(\alpha_i,e,\alpha_i)}=\varnothing$. The obtained contradiction implies that cardinal $\lambda_i$ is finite.
\end{proof}

Theorem~\ref{theorem-2.12} and Proposition~\ref{proposition-3.4} imply the following:

\begin{theorem}\label{theorem-3.5}
Let $S$ be a Hausdorff primitive inverse feebly compact topological semigroup and $S$ be an orthogonal sum of the family $\{B_{\lambda_{i}}(G_i)\colon i\in\mathscr{I}\}$ of topological Brandt semigroups with zeros. Then the following assertions hold:
\begin{enumerate}
  \item[$(i)$] every cardinal $\lambda_i$ is finite;
  \item[$(ii)$] every maximal subgroup of $S$ is open-and-closed subset of $S$ and hence is feebly compact;
  \item[$(iii)$] for every $i\in\mathscr{I}$ the maximal Brandt semigroup $B_{\lambda_{i}}(G_i)$ is a feebly compact topological Brandt $\lambda$-extension of the feebly compact paratopological group $G_i$;
  \item[$(iv)$] if $\mathscr{B}_{(\alpha_i,e_{i},\alpha_i)}$ is a base of the topology at the unity $(\alpha_i,e_{i},\alpha_i)$ of a maximal non-zero subgroup $(G_i)_{\alpha_i,\alpha_i}$ of $S$, $i\in\mathscr{I}$, such that $U\subseteq (G_i)_{\alpha_i,\alpha_i}$ for any $U\in \mathscr{B}_{(\alpha_i,e_{i},\alpha_i)}$, then the family
      \begin{equation*}
        \mathscr{B}_{(\beta_i,x,\gamma_i)}=\left\{(\beta_i,x,\alpha_i)\cdot U\cdot (\alpha_i,e_i,\gamma_i)\colon U\in \mathscr{B}_{(\alpha_i,e_{i},\alpha_i)} \right\}
      \end{equation*}
      is a base of the topology at the point $(\beta_i,x,\gamma_i)\in (G_i)_{\beta_i,\gamma_i}\subseteq B_{\lambda_{i}}(G_i)$, for all $\beta_i,\gamma_i\in \lambda_i$.
\end{enumerate}
If in addition the topological space $S$ is semiregular then
\begin{enumerate}
  \item[$(v)$] the family
      \begin{equation*}\label{eq-3.1}
      \begin{split}
        \mathscr{B}_{0}= \big\{S\setminus& \big((G_{i_1})_{\alpha_{i_1},\beta_{i_1}}\cup \cdots\cup (G_{i_k})_{\alpha_{i_k},\beta_{i_k}}\big)\colon i_1,\ldots,i_k\in\mathscr{I}, \alpha_{i_k},\beta_{i_k}\in\lambda_{i_k},\\
          & k\in\mathbb{N},  \{(\alpha_{i_1},\beta_{i_1}),\ldots,(\alpha_{i_k},\beta_{i_k})\} \hbox{~is finite}\big\}
      \end{split}
      \end{equation*}
      is a base of the topology at zero of $S$.
\end{enumerate}
\end{theorem}

The following example shows that statement $(v)$ of Theorem~\ref{theorem-3.5} does not necessarily hold when the semigroup $S$ is functionally Hausdorff and countably pracompact but it is not semiregular.

\begin{example}\label{example-3.6}
In \cite[Example 3]{Ravsky-arxiv1003.5343v5} a functionally Hausdorff $\omega$-precompact first countable paratopological group $(G,\tau_R)$ is constructed such that each power of $(G,\tau_R)$ is countably pracompact but $(G,\tau_R)$ is not a topological group. Moreover, the group $(G,\tau_R)$ contains an open dense subsemigroup $S$. Let $\mathscr{I}$ be an infinite set of indices. For any $i\in\mathscr{I}$ let $\lambda_i$ be any finite cardinal $\geqslant 1$. Let $B_{\lambda_i}(G)$ be the algebraic Brandt $\lambda_i$-extension of the algebraic group $G$ for each $i\in\mathscr{I}$. Put $R(G,\{\lambda_i\}_{\i\in\mathscr{I}})=\sum_{i\in\mathscr{I}}B_{\lambda_i}(G)$. Also for any subset $C$ of the group $G$ and all $i,i_1,\ldots,i_k\in\mathscr{I}$, $k\in\mathbb{N}$, put
\begin{equation*}
\begin{split}
  B_{\lambda_i}(C) &  =\{0\}\cup\left\{(\alpha_i,x,\beta_i)\in B_{\lambda_i}(G)\colon x\in C, \alpha_i,\beta_i\in\lambda_i\right\}, \qquad R(C,\{\lambda_i\}_{\i\in\mathscr{I}})=\sum_{i\in\mathscr{I}}B_{\lambda_i}(C) \\
    \hbox{and}&\qquad U(i_1,\ldots,i_k)=R(S,\{\lambda_i\}_{\i\in\mathscr{I}})\setminus \big((B_{\lambda_1}(S))^*\cup\cdots\cup (B_{\lambda_k}(S))^*\big).
\end{split}
\end{equation*}

We define the topology $\tau_{RB}$ on $R(G,\{\lambda_i\}_{\i\in\mathscr{I}})$ in the following way:
\begin{enumerate}
  \item[$(i)$] if $\mathscr{B}_{e}$ is a base of the topology $\tau_{R}$ at the unit $e$ of the group $G$ then the family
      \begin{equation*}
        \mathscr{B}_{(\beta_i,x,\gamma_i)}=\left\{(\beta_i,xU,\gamma_i)\colon U\in \mathscr{B}_{e} \right\}
      \end{equation*}
      is a base of the topology $\tau_{RB}$ at the point $(\beta_i,x,\gamma_i)\in G_{\beta_i,\gamma_i}\subseteq B_{\lambda_{i}}(G_i)$, for all $\beta_i,\gamma_i\in \lambda_i$;
  \item[$(ii)$] the family $\mathscr{B}_{0}= \left\{U(i_1,\ldots,i_k)\colon i_1,\ldots,i_k\in\mathscr{I}\right\}$
      is a base of the topology at zero of $R(G,\lambda_i,\mathscr{I})$.
\end{enumerate}

It is obvious that $(R(G,\{\lambda_i\}_{\i\in\mathscr{I}}),\tau_{RB})$ is a Hausdorff topological space. Since $S$ is a dense open subsemigroup of $(G,\tau_R)$ we conclude that $(R(G,\{\lambda_i\}_{\i\in\mathscr{I}}),\tau_{RB})$ is not semiregular. Also, since the space $(G,\tau_R)$ is functionally Hausdorff and $G_{\beta_i,\gamma_i}$ is an open-and-closed subspace of $(R(G,\{\lambda_i\}_{\i\in\mathscr{I}}),\tau_{RB})$, for all $\beta_i,\gamma_i\in \lambda_i$, the space $(R(G,\{\lambda_i\}_{\i\in\mathscr{I}}),\tau_{RB})$ is functionally Hausdorff too.

Now, the definition of the semigroup $R(G,\{\lambda_i\}_{\i\in\mathscr{I}})$ implies that
\begin{equation*}
    U(i_1,\ldots,i_k)\cdot B_{\lambda_m}(G)=B_{\lambda_m}(G)\cdot U(i_1,\ldots,i_k)=\{0\},
\end{equation*}
for each $i_m\in\{i_1,\ldots,i_k\}$ and $U(i_1,\ldots,i_k)\cdot U(i_1,\ldots,i_k)\subseteq U(i_1,\ldots,i_k)$ for all $i_1,\ldots,i_k\in\mathscr{I}$, $k\in\mathbb{N}$, because $S$ is a subsemigroup of the group $G$. This and the continuity of multiplication in $(G,\tau_R)$ imply that multiplication in $(R(G,\{\lambda_i\}_{\i\in\mathscr{I}})$ is continuous.

We claim that the topological space $(R(G,\{\lambda_i\}_{\i\in\mathscr{I}}),\tau_{RB})$ is countably pracompact. Indeed, there exists a set $A\subset S\subset G$ such that
$A$ is dense in the space $(G,\tau_R)$ and this space is countably compact at $A$
~\cite[Example 3]{Ravsky-arxiv1003.5343v5}. Then the set $R(A,\{\lambda_i\}_{\i\in\mathscr{I}})$ is dense in $R(G,\{\lambda_i\}_{\i\in\mathscr{I}})$. We claim that the space $R(G,\{\lambda_i\}_{\i\in\mathscr{I}})$ is countably compact at $R(A,\lambda_i,\mathscr{I})$. Indeed, let $A'$ be an arbitrary countable infinite subset of the set  $R(A,\{\lambda_i\}_{\i\in\mathscr{I}})$. If $0$ is not an accumulation point of the set $A'$ then there exist indices $i_1,\ldots,i_k\in\mathscr{I}$ such that the set $U(i_1,\ldots,i_k)\cap A'$ is finite. Since $A\subset S$ then $A'\subset
R(A,\{\lambda_i\}_{\i\in\mathscr{I}})\subset R(S,\{\lambda_i\}_{\i\in\mathscr{I}})$ and the set $A'\cap
\big((B_{\lambda_1}(S))^*\cup\cdots\cup (B_{\lambda_k}(S))^*\big)\subset A'\setminus U(i_1,\ldots,i_k)$ is infinite. Since for each $1\le j\le k$ the cardinal $\lambda_{i_j}$ is finite, there exists an index $1\le j\le k$ and elements $\alpha,\beta\in\lambda_{i_j}$ such that the intersection $A'\cap S_{\alpha,\beta}\subset B_{\lambda_{i_j}}(A)\subset B_{\lambda_{i_j}}(G)$ is infinite. Since the space $G$ is countably compact at $A$, $B_{\lambda_{i_j}}(G)$ is countable compact at $B_{\lambda_{i_j}}(A)$. Therefore the set $A'\cap S_{\alpha,\beta}$ has an accumulation point in $B_{\lambda_{i_j}}(G)$.
\end{example}

Unlike functional Hausdorffness the quasiregularity guaranties stronger properties of primitive inverse feebly compact topological semigroups and this follows from the next two propositions.

\begin{theorem}\label{theorem-3.7}
Let $S$ be a quasiregular primitive inverse feebly compact topological semigroup and $S$ be the orthogonal sum of the family $\{B_{\lambda_{i}}(G_i)\colon i\in\mathscr{I}\}$ of topological Brandt semigroups with zeros. Then the family
      \begin{equation*}\label{eq-3.2}
      \begin{split}
        \mathscr{B}_{0}= \big\{S\setminus& \big((G_{i_1})_{\alpha_{i_1},\beta_{i_1}}\cup \cdots\cup (G_{i_k})_{\alpha_{i_k},\beta_{i_k}}\big)\colon i_1,\ldots,i_k\in\mathscr{I}, \alpha_{i_k},\beta_{i_k}\in\lambda_{i_k},\\
          & k\in\mathbb{N},  \{(\alpha_{i_1},\beta_{i_1}),\ldots,(\alpha_{i_k},\beta_{i_k})\} \hbox{~is finite}\big\}
      \end{split}
      \end{equation*}
is a base of the topology at zero of $S$.
\end{theorem}

\begin{proof}
Suppose for a contradiction that there exists an open subset $W\ni 0$ of $S$ such that $U\nsubseteq W$ for any $U\in\mathscr{B}_{0}$. There exists an open neighbourhood $V\subseteq W$ of zero in $S$ such that $V\cdot V\cdot V\subseteq W$.
Since every non-zero maximal subgroup of $S$ is an open-and-closed subset of $S$ and the space $S$ is feebly compact, there exist finitely many indices $i_1,\ldots,i_k\in\mathscr{I}$ such that $V\cap (S\setminus\big((B_{\lambda_1}(S))^*\cup\cdots\cup (B_{\lambda_k}(S))^*\big))$ is a dense open subset of the space $S\setminus\big((B_{\lambda_1}(S))^*\cup\cdots\cup (B_{\lambda_k}(S))^*\big)$. Then every non-zero maximal subgroup of $S$ is a quasi-regular space and hence by Proposition~3 of \cite{Ravsky-arxiv1406.2001} (see also \cite{RavskyReznichenko2002}) every maximal subgroup of $S$ is a topological group. Now, Proposition~2.5 of \cite{GutikPavlyk2013??} implies that
\begin{equation*}
V\cdot V\cdot V\supseteq S\setminus\left((B_{\lambda_1}(S))^*\cup\cdots\cup (B_{\lambda_k}(S))^*\right)\not\subseteq W.
\end{equation*}
The obtained contradiction implies the required conclusion.
\end{proof}

Since by Proposition~3 of \cite{Ravsky-arxiv1406.2001} inversion on a quasiregular feebly compact paratopological group is continuous, Proposition~\ref{proposition-2.20}, Theorems~\ref{theorem-3.5} and \ref{theorem-3.7} imply the following corollary:

\begin{corollary}\label{corollary-3.8}
Inversion on a quasi-regular primitive inverse feebly compact topological semigroup $S$ is continuous and hence $S$ is Tychonoff.
\end{corollary}

\begin{remark}\label{remark-3.9}
Example~1 of \cite{BanakhGutik2004} shows that inversion on a quasi-regular inverse countably compact topological semigroup in which maximal subgroups are topological groups is not continuous. Also Corollary~\ref{corollary-3.8} and Proposition~2.8 from \cite{GutikPavlyk2013??} imply that a quasi-regular primitive inverse feebly compact topological semigroup is Tychonoff.
\end{remark}

Also, Corollary~\ref{corollary-3.8} implies

\begin{corollary}\label{corollary-3.8a}
Every quasi-regular feebly compact Brandt topological semigroup is a Tychonoff topological inverse semigroup.
\end{corollary}

Theorem~\ref{theorem-2.11} implies the following:

\begin{theorem}\label{theorem-3.10}
Let $S$ be a Hausdorff primitive inverse countably compact topological semigroup and $S$ be the orthogonal sum of the family $\{B_{\lambda_{i}}(G_i)\colon i\in\mathscr{I}\}$ of topological Brandt semigroups with zeros. Then  the family
      \begin{equation*}\label{eq-3.1}
      \begin{split}
        \mathscr{B}_{0}= \big\{S\setminus& \big((G_{i_1})_{\alpha_{i_1},\beta_{i_1}}\cup \cdots\cup (G_{i_k})_{\alpha_{i_k},\beta_{i_k}}\big)\colon i_1,\ldots,i_k\in\mathscr{I}, \alpha_{i_k},\beta_{i_k}\in\lambda_{i_k},\\
          & k\in\mathbb{N},  \{(\alpha_{i_1},\beta_{i_1}),\ldots,(\alpha_{i_k},\beta_{i_k})\} \hbox{~is finite}\big\}
      \end{split}
      \end{equation*}
      is a base of the topology at zero of $S$.
\end{theorem}

Definition~\ref{def1}, Theorem~\ref{theorem-3.10} and arguments presented in the proof of Theorem~\ref{theorem-2.11} imply the following corollary:

\begin{corollary}\label{corollary-3.11}
Inversion on a Hausdorff primitive inverse countably compact topological semigroup $S$ is continuous if and only if every maximal subgroup of $S$ is a topological group.
\end{corollary}

\begin{remark}\label{remark-3.12}
The second listed author, using a result of P. Koszmider, A. Tomita and S. Watson \cite{KoszmiderTomitaWatson2000}, constructed under \textsf{MA} an example of a Hausdorff countably compact paratopological group failing to be a topological group \cite{Ravsky2003, Ravsky-arxiv1003.5343v5}.
\end{remark}

\section{Products of feebly compact inverse primitive semitopological semigroups and their Stone-\v{C}ech compactifications}

The counterparts of the following four statements for Tychonoff spaces
are proved in \cite[Section~3.10]{Engelking1989}. But because the proofs which are based
on continuous functions are not applicable for our case,
we present straightforward proofs here.

%The following four statements is well-known folklore results and their analogues for
%Tychonoff topological spaces proved in \cite[Section~3.10]{Engelking1989}. Since we could not
%find a precise reference, we give short proofs for convenience of the reader.

\begin{proposition}\label{proposition-2.15-0}
Let $X$ be a feebly compact topological space and $Y$ be a
sequentially compact topological space. Then $X\times Y$ is feebly compact.
\end{proposition}
\begin{proof} We have to prove that any infinite family $\left\{U_n\colon n\in\mathbb N\right\}$
of non-empty open subsets of the space $X\times Y$ is not locally finite. For this purpose we
shall find a point $(x,y)\in X\times Y$ such that every open neighbourhood of $(x,y)$ intersects infinitely many elements of the family $\left\{U_n\colon n\in\mathbb N\right\}$.
Let $n$ be a positive integer.
There exist non-empty open subsets
$V_n\subset X$ and $W_n\subset Y$ such that $V_n\times W_n\subset U_n$.
Choose a point $y_n\in\ W_n$.
Since the space $Y$ is sequentially compact, the
sequence $\left\{y_{n}:n\in\mathbb N\right\}$ has a
subsequence $\left\{y_{n_k}:k\in\mathbb N\right\}$ converging to a point $y\in Y$.
Since the space $X$ is feebly compact, there exists a point $x\in X$
such that every open neighbourhood of the point $x$ in $X$ intersects $V_{n_k}$ for infinitely many numbers $k$. Then each open neighbourhood of the point $(x,y)\in X\times Y$ intersects $U_n$ for infinitely many numbers $n$.
Hence $(x,y)$ is the required point.
\end{proof}

\begin{proposition}\label{proposition-2.15}
Let $X$ be a Hausdorff feebly compact topological space. Then $X\times Y$ is feebly compact for any feebly compact $k$-space $Y$.
\end{proposition}

\begin{proof}
It suffice to observe that every non-empty open subset of the Cartesian product $X\times Y$ contains an open subset $U\times V$, where $U$ and $V$ are non-empty open subset of $X$ and $Y$, respectively, and then Lemma~3.10.12 of \cite{Engelking1989} implies the statement of the proposition.
\end{proof}

Proposition~\ref{proposition-2.15} implies the following two corollaries.

\begin{corollary}\label{corollary-2.16}
The Cartesian product $X\times Y$ of a feebly compact space $X$ and a compactum $Y$ is feebly compact.
\end{corollary}

\begin{corollary}\label{corollary-2.17}
The Cartesian product $X\times Y$ of a feebly compact space $X$ and a feebly compact sequential space $Y$ is feebly compact.
\end{corollary}

\begin{proposition}\label{proposition-2.18}
Let $S$ be a primitive semitopological inverse semigroup such that every maximal subgroup of $S$ is a feebly compact paratopological (topological) group. Then $S$ is a continuous\footnote{not necessarily a homomorphic image} image of the product $\widetilde{E}_S\times G_S$, where $\widetilde{E}_S$ is a compact semilattice and $G_S$ is a feebly compact paratopological (topological) group provided one of the following conditions holds:
\begin{itemize}
  \item[(1)] $S$ is semiregular and feebly compact;
  \item[(2)] $S$ is Hausdorff and countably compact.
\end{itemize}
\end{proposition}

\begin{proof}
We only consider the case when $S$ is a semiregular feebly compact space and every maximal subgroup of $S$ is a paratopological group because in case (2) the proof is similar.

By Theorem~\ref{theorem-2.12} the topological semigroup $S$ is topologically isomorphic to the orthogonal sum $\sum_{i\in\mathscr{I}}B_{\lambda_{i}}(G_i)$ of topological Brandt $\lambda_i$-extensions $B_{\lambda_i}(G_i)$ of feebly compact paratopological groups $G_i$ in the class of Hausdorff semitopological semigroups for some cardinals $\lambda_i\geqslant 1$ and the family defined by formula (\ref{eq-2.1}) in Theorem~\ref{theorem-2.12}$(iv)$ determines the base of a topology at zero of $S$.

Fix an arbitrary $i\in\mathscr{I}$. Then by Corollary~\ref{corollary-2.9} the space $E(B_{\lambda_i}(G_i))$ is compact. First we consider the case when the cardinal $\lambda_i$ is finite. Suppose that $|E(B_{\lambda_i}(G_i))|=n_i+1$ for some integer $n_i$. Then $\lambda_i=n_i\geqslant1$. On the set $E_i=(\lambda_i\times\lambda_i)\cup\{0\}$, where $0\notin\lambda_i\times\lambda_i$ we define multiplication in the following way
\begin{equation*}
(\alpha,\beta)\cdot(\gamma,\delta)=
    \left\{
       \begin{array}{cl}
         (\alpha,\beta), & \hbox{if~} (\alpha,\beta)=(\gamma,\delta);\\
         0, & \hbox{otherwise,}
       \end{array}
     \right.
\end{equation*}
and $0\cdot(\alpha,\beta)=(\alpha,\beta)\cdot 0=0\cdot 0=0$ for all $\alpha,\beta,\gamma,\delta\in\lambda_i$. Simple verifications show that $E_i$ with this multiplication is a semilattice and every non-zero idempotent of $E_i$ is primitive. If the cardinal $\lambda_i$ is infinite then on the set $E_i=(\lambda_i\times\lambda_i)\cup\{0\}$ we define the semilattice operation in a similar way

We denote by $\widetilde{E}_S$ the orthogonal sum $\sum_{i\in\mathscr{I}}E_i$. It is obvious that $\widetilde{E}_S$ is a semilattice and every non-zero idempotent of $\widetilde{E}_S$ is primitive. We determine on $\widetilde{E}_S$ the topology of the Alexandroff one-point compactification $\tau_A$: all non-zero idempotents of $\widetilde{E}_S$ are isolated points in $\widetilde{E}_S$ and the family
\begin{equation*}
    \mathscr{B}(0)=\big\{U\colon U\ni 0 \;\hbox{ and } \; \widetilde{E}_S\setminus U \; \hbox{ is finite}\big\}
\end{equation*}
is the base of the topology $\tau_A$ at zero $0\in\widetilde{E}_S$. Simple verifications show that $\widetilde{E}_S$ with the topology $\tau_A$ is a Hausdorff compact topological semilattice. Later we denote $(\widetilde{E}_S,\tau_A)$ by $\widetilde{E}_S$.

Let $G_S=\prod_{i\in\mathscr{I}}G_i$ be the direct product of feebly compact paratopological groups $G_i$, $i\in\mathscr{I}$, with the Tychonoff topology. Then Proposition~24 from \cite{Ravsky-arxiv1003.5343v5} implies that $G_S$ is a feebly compact paratopological group. Also by Corollary~\ref{corollary-2.16} we have that the product $\widetilde{E}_S\times G_S$ is a feebly compact space.

For every $i\in\mathscr{I}$ we denote by $\pi_i\colon G_S=\prod_{i\in\mathscr{I}}G_i\rightarrow G_i$ the projection on the $i$-th factor.

Now, for every $i\in\mathscr{I}$ we define the map $f_i\colon E_i\times G_S\rightarrow B_{\lambda_i}(G_i)$  by the formulae $f_i((\alpha,\beta),g)=(\alpha,\pi_i(g),\beta)$ and $f_i(0,g)=0_i$ is zero of the semigroup $B_{\lambda_i}(G_i)$, and put $f=\bigcup_{i\in\mathscr{I}}f_i$. It is obvious that the map $f\colon \widetilde{E}_S\times G_S\rightarrow S$ is well defined. The definition of the topology $\tau_A$ on $\widetilde{E}_S$ implies that for every $((\alpha,\beta),g)\in E_i\times G_i\subseteq \widetilde{E}_S\times G_i$  the set $\{(\alpha,\beta)\}\times G_i$ is open in $\widetilde{E}_S\times G_S$ and hence the map $f$ is continuous at the point $((\alpha,\beta),g)$. Also for every $U(0)=S\setminus\big(B_{\lambda_{i_1}}(G_{i_1})\cup
B_{\lambda_{i_2}}(G_{i_2})\cup\cdots\cup B_{\lambda_{i_n}}(G_{i_n})\big)^*$ the set $f^{-1}(U(0))=\big(\widetilde{E}_S\setminus\left((\lambda_{i_1}\times\lambda_{i_1}) \cup\ldots\cup(\lambda_{i_n}\times\lambda_{i_n})\right)\big)\times G_S$ is open in $\widetilde{E}_S\times G_S$, and hence the map $f$ is continuous.

We observe that in the case when all maximal subgroups of $S$ are topological groups, $G_S=\prod_{i\in\mathscr{I}}G_i$ is a pseudocompact topological group by Comfort--Ross theorem (see Theorem~1.4 in \cite{ComfortRoss1966}).

Also, in the case of a Hausdorff semitopological semigroup $S$ the proof is similar.
\end{proof}

The following result is an extension of the Comfort--Ross Theorem for primitive feebly compact semitopological inverse semigroups.

\begin{theorem}\label{theorem-2.19}
Let $\{S_j\colon j\in\mathscr{J}\}$ be a family of primitive semitopological inverse semigroups such that for each $j\in\mathscr{J}$ the semigroup $S_j$ is either semiregular feebly compact or Hausdorff countably compact, and suppose that each maximal subgroup of $S_j$ a feebly compact paratopological group. Then the direct product $\prod_{j\in\mathscr{J}}S_j$ with the Tychonoff topology is a feebly compact semitopological inverse semigroup.
\end{theorem}

\begin{proof}
Since the direct product of a family of semitopological inverse semigroups is a semitopological inverse semigroup, it is sufficient to show that the space $\prod_{j\in\mathscr{J}}S_j$ is feebly compact. For each $j\in\mathscr{J}$ let $\widetilde{E}_{S_j}$, ${G}_{S_j}$, and $f_j\colon \widetilde{E}_{S_j}\times {G}_{S_j}\rightarrow S_j$ be the semilattice, the group and the map, respectively, defined in the proof of Proposition~\ref{proposition-2.18}. Since the space $\prod_{j\in\mathscr{J}}\left(\widetilde{E}_{S_j}\times {G}_{S_j}\right)$ is homeomorphic to the product $\prod_{j\in\mathscr{J}}\widetilde{E}_{S_j}\times \prod_{j\in\mathscr{J}}{G}_{S_j}$ we conclude that by Theorem~3.2.4 from \cite{Engelking1989}, Corollary~\ref{corollary-2.16} and Proposition~24 from \cite{Ravsky-arxiv1003.5343v5} the space $\prod_{j\in\mathscr{J}}\left(\widetilde{E}_{S_j}\times {G}_{S_j}\right)$ is feebly compact. Now, since the map $\prod_{j\in\mathscr{J}}f_j\colon \prod_{j\in\mathscr{J}}\left(\widetilde{E}_{S_j}\times {G}_{S_j}\right)\rightarrow \prod_{j\in\mathscr{J}}S_j$ is continuous  $\prod_{j\in\mathscr{J}}S_j$ is a feebly compact topological space.
\end{proof}

The proofs of the following two propositions are similar to Proposition~\ref{proposition-2.18} and Theorem~\ref{theorem-2.19}; they generalize  Proposition~2.11 and Theorem~2.12 from \cite{GutikPavlyk2013??}.

\begin{proposition}\label{proposition-3.13}
Let $S$ be a primitive inverse topological semigroup. Then $S$ is a continuous (not necessarily homomorphic) image of the product $\widetilde{E}_S\times G_S$, where $\widetilde{E}_S$ is a compact semilattice and $G_S$ is a feebly compact paratopological group provided one of the following conditions holds:
\begin{itemize}
  \item[(1)] $S$ is semiregular feebly compact;
  \item[(2)] $S$ is Hausdorff countably compact.
\end{itemize}
\end{proposition}

\begin{theorem}\label{theorem-3.14}
Let $\{S_i\colon i\in\mathscr{J}\}$ be a family of primitive inverse semiregular feebly compact (Hausdorff countably compact) topological semigroups. Then the direct product $\prod_{j\in\mathscr{J}}S_j$ with the Tychonoff topology is a feebly compact inverse topological semigroup.
\end{theorem}

Let a Tychonoff topological space $X$ be a topological sum of subspaces $A$ and $B$, i.e., $X=A\bigoplus B$. It is obvious that every continuous map $f\colon A\rightarrow K$ from $A$ into a compact space $K$ (resp., $f\colon B\rightarrow K$ from $B$ into a compact space $K$) extends to a continuous map $\widehat{f}\colon X\rightarrow K$. This implies the following proposition:

\begin{proposition}\label{proposition-4.1}
If a Tychonoff topological space $X$ is the topological sum of subspaces $A$ and $B$, then $\beta X$ is equivalent to the topological sum $\beta A\bigoplus\beta B$.
\end{proposition}

The following theorem follows from Corollary~\ref{corollary-3.8} and Theorem~3.2 of \cite{GutikPavlyk2013??}, and it describes the structure of the Stone-\v{C}ech compactification of a primitive inverse feebly compact quasi-regular topological semigroup.

\begin{theorem}\label{theorem-4.2}
Let $S$ be a primitive inverse feebly compact quasi-regular topological semigroup. Then the Stone-\v{C}ech compactification of $S$ admits the structure of a primitive topological inverse semigroup with respect to which the inclusion mapping of $S$ into $\beta{S}$ is a topological isomorphism. Moreover, $\beta{S}$ is topologically isomorphic to the orthogonal sum $\sum_{i\in\mathscr{I}}B_{\lambda_{i}}(\beta G_i)$ of topological Brandt $\lambda_i$-extensions $B_{\lambda_i}(\beta G_i)$ of compact topological groups $\beta G_i$ in the class of topological inverse semigroups for some finite cardinals $\lambda_i\geqslant 1$.
\end{theorem}

\begin{theorem}\label{theorem-4.3}
Let $S$ be a regular primitive inverse countably compact semitopological semigroup and $S$ be the orthogonal sum of a family $\{B_{\lambda_{i}}(G_i)\colon i\in\mathscr{I}\}$ of semitopological Brandt semigroups with zeros. Suppose that for every $i\in\mathscr{I}$ there exists a maximal non-zero subgroup $(G_i)_{\alpha_i,\alpha_i}$, $\alpha_i\in\lambda_i$, such that at least the one of the following conditions holds:
\begin{itemize}
  \item[(1)] the group $(G_i)_{\alpha_i,\alpha_i}$ is left precompact;
  \item[(2)] the group $(G_i)_{\alpha_i,\alpha_i}$ is left $\omega$-precompact feebly compact;
  \item[(3)] the semigroup $S_{\alpha_i,\alpha_i}= (G_i)_{\alpha_i,\alpha_i}\cup\{0\}$ is a topological semigroup.
\end{itemize}
Then the Stone-\v{C}ech compactification of $S$ admits the structure of a primitive inverse semitopological semigroup with continuous inversion with respect to which the inclusion mapping of $S$ into $\beta{S}$ is a topological isomorphism. Moreover, $\beta{S}$ is topologically isomorphic to the orthogonal sum $\sum_{i\in\mathscr{I}}B_{\lambda_{i}}(\beta G_i)$ of compact topological Brandt $\lambda_i$-extensions $B_{\lambda_i}(\beta G_i)$ of compact topological groups $\beta G_i$ in the class of semitopological semigroups for some cardinals $\lambda_i\geqslant 1$.
\end{theorem}

\begin{proof}
By Theorem~\ref{theorem-2.11}, the semigroup $S$ is topologically isomorphic to the orthogonal sum $\sum_{i\in\mathscr{I}}B_{\lambda_{i}}(G_i)$ of topological Brandt $\lambda_i$-extensions $B_{\lambda_i}(G_i)$ of countably compact paratopological groups $G_i$ in the class of semitopological semigroups for some cardinals $\lambda_i\geqslant 1$, such that any non-zero $\mathscr{H}$-class of $S$ is an open-and-closed subset of $S$, and the family $\mathscr{B}(0)$ defined by formula (\ref{eq-2.1}) in Theorem~\ref{theorem-2.12}$(iv)$  determines a base of the topology at zero $0$ of $S$. Since the space $S$ regular and any non-zero $\mathscr{H}$-class of $S$ is an open-and-closed subset of $S$, every maximal subgroup of $S$ is a topological group \cite[Proposition~3]{Ravsky-arxiv1003.5343v5}. Hence $S$ is topologically isomorphic to the orthogonal sum $\sum_{i\in\mathscr{I}}B_{\lambda_{i}}(G_i)$ of topological Brandt $\lambda_i$-extensions $B_{\lambda_i}(G_i)$ of countably compact topological groups $G_i$ in the class of semitopological semigroups for some cardinals $\lambda_i\geqslant 1$. Then by Proposition~\ref{proposition-2.20} the semigroup $S$ is Tychonoff, and hence the Stone-\v{C}ech compactification of $S$ exists.

By Theorem~\ref{theorem-2.19}, $S\times S$ is a pseudocompact topological space. Now by Theorem~1 of \cite{Glicksberg1959}, we have that $\beta(S\times S)$ is equivalent to $\beta S\times \beta S$, and hence by Theorem~1.1 of~\cite{BanakhDimitrova2010}, $S$ is a subsemigroup of the compact semitopological semigroup $\beta S$.

By Proposition~\ref{proposition-4.1} for every non-zero $\mathscr{H}$-class $(G_i)_{k,l}$, $k,l\in\lambda_i$, we have that $\operatorname{cl}_{\beta S}((G_i)_{k,l})$ is equivalent to $\beta(G_i)_{k,l}$, and hence it is equivalent to $\beta G_i$. Therefore we may naturally consider the space $\sum_{i\in\mathscr{I}}B_{\lambda_{i}}(\beta G_i)$ as a subspace of the space $\beta S$. Suppose that $\sum_{i\in\mathscr{I}}B_{\lambda_{i}}(\beta G_i)\neq \beta S$. We fix an arbitrary $x\in\beta S\setminus \sum_{i\in\mathscr{I}}B_{\lambda_{i}}(\beta G_i)$. Then Hausdorffness of $\beta S$ implies that there exist open neighbourhoods $V(x)$ and $V(0)$ of the points $x$ and the zero $0$ in $\beta S$, respectively, and there exist finitely many indices $i_1,\ldots,i_k\in\mathscr{I}$ and finitely many pairs of indices $(\alpha_{i_1},\beta_{i_1}),\ldots,(\alpha_{i_k},\beta_{i_k})$ such that $V(0)\cap\beta S\supseteq S\setminus  \big((G_{i_1})_{\alpha_{i_1},\beta_{i_1}}\cup \cdots\cup (G_{i_k})_{\alpha_{i_k},\beta_{i_k}}\big)$. Then we have that
\begin{equation*}
V(x)\cap S \subseteq \big((G_{i_1})_{\alpha_{i_1},\beta_{i_1}}\cup \cdots\cup (G_{i_k})_{\alpha_{i_k},\beta_{i_k}}\big)\subseteq \big((\beta G_{i_1})_{\alpha_{i_1},\beta_{i_1}}\cup \cdots\cup (\beta G_{i_k})_{\alpha_{i_k},\beta_{i_k}}\big).
\end{equation*}
But this contradicts that $x$ is an accumulation point of $\sum_{i\in\mathscr{I}}B_{\lambda_{i}}(\beta G_i)$ in $\beta S$, which does not belong to  $\sum_{i\in\mathscr{I}}B_{\lambda_{i}}(\beta G_i)$, because
$(\beta G_{i_1})_{\alpha_{i_1},\beta_{i_1}}\cup \cdots\cup (\beta G_{i_k})_{\alpha_{i_k},\beta_{i_k}}$ is a compact subset of $\beta S$.
\end{proof}

Recall~\cite{DeLeeuwGlicksberg1961} that the \emph{Bohr compactification} of a semitopological semigroup $S$ is a~pair $(\textsf{b}, \mathbb{B}(S))$ such that $\mathbb{B}(S)$ is a compact semitopological semigroup, $\textsf{b}\colon S\to \mathbb{B}(S)$ is a continuous homomorphism, and if $g\colon S\to T$ is a continuous homomorphism of $S$ into a compact semitopological semigroup $T$, then there exists a unique continuous homomorphism $f\colon \mathbb{B}(S)\to T$ such that the diagram
\begin{equation*}
\xymatrix{ S\ar[rr]^{\textsf{b}}\ar[dr]_g && \mathbb{B}(S)\ar[ld]^f\\
& T &}
\end{equation*}
commutes. In the sequel, similar to that in General Topology, by the Bohr compactification of a semitopological semigroup $S$ we shall mean not only pair $(\textsf{b},\mathbb{B}(S))$ but also the compact semitopological semigroup $\mathbb{B}(S)$.

The definitions of the Stone-\v{C}ech compactification and the Bohr compactification, and Theorem~\ref{theorem-4.3} imply the following corollary:

\begin{corollary}\label{corollary-4.4}
Let $S$ be a Hausdorff primitive inverse countably compact semitopological semigroup such that every maximal subgroup of $S$ is a pseudocompact topological group and $S$ be the orthogonal sum of a family $\{B_{\lambda_{i}}(G_i)\colon i\in\mathscr{I}\}$ of semitopological Brandt semigroups with zeros. Then the Bohr compactification of $S$ admits the structure of a primitive inverse semitopological semigroup with continuous inversion with respect to which the inclusion mapping of $S$ into $(\textsf{b},\mathbb{B}(S))$ is a topological isomorphism. Moreover, $(\textsf{b},\mathbb{B}(S))$ is topologically isomorphic to the orthogonal sum $\sum_{i\in\mathscr{I}}B_{\lambda_{i}}(\beta G_i)$ of topological Brandt $\lambda_i$-extensions $B_{\lambda_i}(\beta G_i)$ of compact topological groups $\beta G_i$ in the class of semitopological semigroups for some cardinals $\lambda_i\geqslant 1$.
\end{corollary}
%%%%%%%%%%%%%%%%%%%%%%%%%%%%%%%%%%%%%%%%%%%%%%%%%%%%%%%%%%%

\section*{Acknowledgements}

The authors are grateful to the referee for useful comments and suggestions.

%%%%%%%%%%%%%%%%%%%%%%%%%%%%%%%%%%%%%%%%%%%%%%%%%%%%%%%%%%%%%%%

\end{document}